\documentclass{amsart}
\usepackage{MnSymbol}
\usepackage[mathscr]{euscript}
\usepackage{scalerel}
\usepackage{colonequals}

\usepackage{hyperref}
\hypersetup{
    colorlinks,
    citecolor=black,
    filecolor=black,
    linkcolor=black,
    urlcolor=black
}
\usepackage{comment}

\DeclareFontFamily{OT1}{pzc}{}
\DeclareFontShape{OT1}{pzc}{m}{it}{<-> s * [1.100] pzcmi7t}{}
\DeclareMathAlphabet{\mathpzc}{OT1}{pzc}{m}{it}

    \usepackage{etoolbox}
    \patchcmd{\section}{\scshape}{\large\bfseries}{}{}
    \makeatletter
    \renewcommand{\@secnumfont}{\bfseries}
    \makeatother

\usepackage{tikz-cd}
\tikzcdset{arrow style=math font}
\tikzcdset{scale cd/.style={every label/.append style={scale=#1},
    cells={nodes={scale=#1}}}}
\usepackage{tikz}
\usetikzlibrary{arrows}
\usepackage{caption}
\usepackage{float}

\usepackage[normalem]{ulem}

\numberwithin{equation}{section}
\newtheorem{theorem}{Theorem}[section]
\newtheorem*{theorem*}{Theorem}
\newtheorem{corollary}[theorem]{Corollary}
\newtheorem{lemma}[theorem]{Lemma}
\newtheorem{proposition}[theorem]{Proposition}
\theoremstyle{definition}

\newtheorem*{question*}{Question}
\newtheorem*{conjecture*}{Conjecture}

\newtheorem{definition}[theorem]{Definition}
\newtheorem{remark}[theorem]{Remark}
\newtheorem{example}[theorem]{Example}

\def\mono{\rightarrowtail}
\def\epi{\twoheadrightarrow}

\def\Ker{\mathrm{Ker}}
\def\Im{\mathrm{Im}}

\def\KK{\mathbb{K}}
\def\ZZ{\mathbb{Z}}
\def\MH{\mathrm{MH}}
\def\NN{\mathscr{N}}

\def\FF{\mathcal{F}}
\def\GG{\mathbb{G}}
\def\Hom{\mathrm{Hom}}
\def\Tor{\mathrm{Tor}}
\def\Ext{\mathrm{Ext}}
\def\MH{\mathrm{MH}}
\def\MC{\mathrm{MC}}
\def\Mag{\mathrm{Mag}}
\def\tP{\mathtt{P}}
\def\tS{\mathtt{S}}
\def\tL{\mathtt{L}}
\def\tN{\mathtt{N}}
\def\RR{\mathcal{R}}

\def\cN{\mathcal{N}}
\def\Lk{\mathrm{Lk}}

\let\oldtocsection=\tocsection 
\let\oldtocsubsection=\tocsubsection 
\renewcommand{\tocsection}[2]{\hspace{0mm}\oldtocsection{#1}{#2}}
\renewcommand{\tocsubsection}[2]{\hspace{1em}\oldtocsubsection{#1}{#2}}

\title[Diagonal digraphs, Koszul algebras and homology spheres]{On diagonal digraphs, Koszul algebras and triangulations of homology spheres}

\author{Sergei O. Ivanov} 
\address{
Beijing Key Laboratory of Topological Statistics and Applications for Complex Systems, Beijing Institute of Mathematical Sciences and Applications (BIMSA), Beijing 101408, China.}
\email{ivanov.s.o.1986@gmail.com, ivanov.s.o.1986@bimsa.cn}

\author{Lev Mukoseev} 
\address{
Saint Petersburg State University, 7/9 Universitetskaya Emb., 199034, Saint Petersburg, Russia.}
\email{la.mukoseev@gmail.com}

\begin{document}

\begin{abstract}  
We study magnitude homology of digraphs, with a particular focus on diagonal digraphs, i.e., digraphs whose magnitude homology is concentrated on the diagonal. For any digraph $G$, we provide a complete description of the second magnitude homology $\MH_{2,k}(G)$. This allows us to define a combinatorial condition, denoted by $(\mathcal{V}_\ell)$, which is equivalent to the vanishing of ${\rm MH}_{2,k}(G,\mathbb{Z})$ for all $k>\ell$. In particular, diagonal digraphs satisfy $(\mathcal{V}_2)$. As a corollary, we obtain that the 2-dimensional CW-complex obtained from a diagonal undirected graph by attaching 2-cells to all squares and triangles of the graph is simply connected. We also give an interpretation of diagonality in terms of Koszul algebras: a digraph $G$ is diagonal if and only if the distance algebra $\sigma G$ is Koszul over any field, and if and only if $G$ satisfies $(\mathcal{V}_2)$ and the path cochain algebra $\Omega^\bullet(G)$ is Koszul over any field. To provide a source of examples of digraphs, we study the extended Hasse diagram $\hat G_K$ of a pure simplicial complex $K$. For a triangulation $K$ of a topological manifold $M$, we express the non-diagonal part of the magnitude homology of $\hat G_K$ in terms of the homology of $M$. As a corollary, we obtain that if $K$ is a triangulation of a closed manifold $M$, then $\hat G_K$ is diagonal if and only if $M$ is a homology sphere.
\end{abstract}

\maketitle

\tableofcontents

\section{Introduction}
Hepworth and Willerton in \cite{hepworth2017categorifying} introduced magnitude homology of a graph as a categorification of the magnitude of a graph studied by Leinster \cite{leinster2019magnitude}. This concept was later generalized by Leinster and Shulman to arbitrary generalized metric spaces and to enriched categories with some additional data \cite{leinster2021magnitude}. Hepworth \cite{hepworth2022magnitude} developed the theory of magnitude cohomology. Kaneta and Yoshinaga pointed out the complexity of magnitude homology of graphs by showing that the homology of any pure simplicial complex can be embedded into the magnitude homology of some graph \cite{kaneta2021magnitude}. Asao studied this theory in the setting of digraphs and showed that it is related to path homology theory \cite{asao2023magnitude}. Asao together with the first-named author proved that the magnitude homology and cohomology of a digraph $G$ can be presented in terms of the Tor and Ext functors over a graded algebra $\sigma G$, which we call the distance algebra of $G$ \cite{asao2024magnitude}.

Hepworth and Willerton call a graph diagonal if its integral magnitude homology groups are concentrated in the diagonal degrees. They provide many examples: trees, complete graphs, the icosahedral graph (see also \cite{gu2018graph}), and the join of two arbitrary graphs. The box product of two diagonal graphs is also diagonal; in particular, hypercubes are diagonal. Median graphs, which are retracts of hypercubes (see \cite{bandelt1984retracts}, \cite{bottinelli2021magnitude}), are also diagonal. Consequently, square graphs are diagonal \cite{seemann2023planar}. Diagonal digraphs can be defined similarly.

This article can be roughly divided into four parts. In the first part, we show that the magnitude homology of a digraph is isomorphic to a subquotient of its path algebra. This description, reminiscent of Gruenberg's description of group homology, is referred to as the Gruenberg formula. In the second part, we use the Gruenberg formula to provide a complete description of the second magnitude homology $\MH_{2,\ell}(G)$ for any $\ell$ and any finite digraph $G$. From this, we derive a necessary condition for a digraph to be diagonal.

The third part connects the diagonality of a digraph to representation theory through the theory of Koszul algebras. Finally, in the fourth part, inspired by the work of Kaneta and Yoshinaga, we study digraphs associated with simplicial complexes, which we call extended Hasse diagrams. We examine conditions under which a digraph corresponding to a triangulation of a manifold is diagonal.

\subsection{The Gruenberg formula}

In \cite{gruenberg1960resolutions}, Gruenberg shows that any presentation of a group $\mathcal{G} \cong \mathcal{F}/\mathcal{R}$ defines an explicit projective resolution of the trivial module, described in terms of the ideals $\mathbf{f}$ and $\mathbf{r}$ of the free group ring $\mathbb{Z}[\mathcal{F}]$. Here, $\mathbf{f}$ denotes the augmentation ideal, and $\mathbf{r}$ is the kernel of the map $\mathbb{Z}[\mathcal{F}] \twoheadrightarrow \mathbb{Z}[\mathcal{G}]$. Since the homology of $\mathcal{G}$ can be presented as Tor functors,
\begin{equation}
H_n(\mathcal{G}) = {\rm Tor}_n^{\mathbb{Z}[\mathcal{G}]}(\mathbb{Z}, \mathbb{Z}),
\end{equation}
this resolution was used to obtain the following isomorphisms for group homology {\cite[\S 3.7]{gruenberg2006cohomological}}:
\begin{equation}
H_{2n}(\mathcal{G}) \cong \frac{\mathbf{r}^n \cap \mathbf{f} \mathbf{r}^{n-1} \mathbf{f}}{\mathbf{f} \mathbf{r}^n + \mathbf{r}^n \mathbf{f}}, \hspace{1cm}
H_{2n+1}(\mathcal{G}) \cong \frac{\mathbf{f} \mathbf{r}^n \cap \mathbf{r}^n \mathbf{f}}{\mathbf{r}^{n+1} + \mathbf{f} \mathbf{r}^n \mathbf{f}}.
\end{equation}

Apparently, the idea of constructing such a resolution first appeared in the article by Eilenberg, Nagao, and Nakayama \cite{eilenberg1956dimension} in 1956. However, it was Gruenberg who popularized this resolution in group homology theory in 1960. This idea was also used by Bachmann \cite{bachmann1972gruenberg} for augmented algebras in 1972 and by Govorov \cite[Lemma 1]{govorov1973dimension} for graded algebras in 1973. For the case of finite-dimensional algebras, the same idea appeared in the work of Bongartz \cite{bongartz1983algebras} and Butler and King \cite{butler1999minimal}. We establish a general version of the Gruenberg formula that covers all of these statements (Theorem \ref{theorem:Gruenberg_formulas}). In this general version of the theorem, we replace the algebra $\mathbb{Z}[\mathcal{F}]$ with an arbitrary quasi-free algebra.

For a finite digraph $G$ and a commutative ring $\KK$, we consider the  distance algebra $\sigma G$, which is a graded algebra defined as a quotient of the path algebra:
\begin{equation}
\sigma G = \KK G / R,
\end{equation}
where $R$ is a homogeneous ideal generated by differences of shortest paths having the same initial and terminal vertices, and by all long paths (i.e., non-shortest paths). In \cite[Th.~6.2]{asao2024magnitude}, Asao and the first-named author proved that magnitude homology with coefficients in $\KK$ can be regarded as the Tor functor over the distance algebra:
\begin{equation}\label{eq:intro:derived}
\MH_{n,\ell}(G) \cong \Tor^{\sigma G}_{n,\ell}(S, S),
\end{equation}
where $S = (\sigma G)_0$. We also denote by $J$ the ideal of $\KK G$ generated by all arrows. Using the Tor description of magnitude homology, the fact that $\KK G$ is a quasi-free algebra, and our general version of the Gruenberg formula, we prove the following formulas for magnitude homology (Theorem~\ref{theorem:Gruenberg_for_magnitude}):
\begin{equation}
\MH_{2n,*}(G) \cong \frac{R^n \cap J R^{n-1} J}{J R^n + R^n J}, \hspace{1cm}
\MH_{2n+1,*}(G) \cong \frac{J R^n \cap R^n J}{R^{n+1} + J R^n J}.
\end{equation}

\subsection{The second magnitude homology group}

For the second magnitude homology group, the Gruenberg formula takes the form:
\begin{equation}\label{eq:MH_2_ell}
\MH_{2,*}(G) \cong \frac{R}{J R + R J}.
\end{equation}
Asao proves that $\MH_{2,\ell}(G)$ is the zeroth homology group of a pair of simplicial complexes \cite{asao2021geometric}. In particular, $\MH_{2,\ell}(G)$ is a free $\KK$-module. The formula \eqref{eq:MH_2_ell} allows us to describe its basis. This provides a necessary and sufficient condition for the vanishing of ${\rm MH}_{2,k}(G)$ for $k > \ell$. We denote this condition by $(\mathcal{V}_\ell)$ and describe it below.

In order to describe the property $(\mathcal{V}_\ell)$, we need to introduce additional terminology. We say that two paths $p = (p_0, \dots, p_n)$ and $q = (q_0, \dots, q_m)$ in $G$ connect the same vertices if $p_0 = q_0$ and $p_n = q_m$. An equivalence relation $\sim$ on the set of paths is called a \emph{congruence} if $p \sim \tilde{p}$ implies that $p$ and $\tilde{p}$ connect the same vertices and $q p q' \sim q \tilde{p} q'$ for any $q$ and $q'$ such that the concatenation is defined. The \emph{$\ell$-short congruence} is the minimal congruence on the set of paths such that any two shortest paths connecting the same two vertices at distance at most $\ell$ are equivalent. A path that is not shortest is called long.

A long path is called \emph{$\ell$-reducible} if it contains a long subpath of length at most $\ell$. A long path is called \emph{$\ell$-quasi-reducible} if it is $\ell$-shortly congruent to an $\ell$-reducible path. Then $\MH_{2,k}(G) = 0$ for any $k > \ell$ if and only if the following condition is satisfied (Theorem~\ref{th:vanishing}):

\begin{itemize}
\item[$(\mathcal{V}_\ell)$] \it Any two shortest paths connecting the same vertices are $\ell$-shortly congruent, and any long path is $\ell$-quasi-reducible.
\end{itemize}

A long path is called minimal if all of its proper subpaths are shortest. To verify the condition $(\mathcal{V}_\ell)$ for a given digraph, it is sufficient to consider the shortest paths and minimal long paths.

Note that any diagonal digraph satisfies the condition $(\mathcal V_2)$, but this condition is not equivalent to diagonality. For instance, the digraphs (an undirected edge means two arrows in opposite directions) 
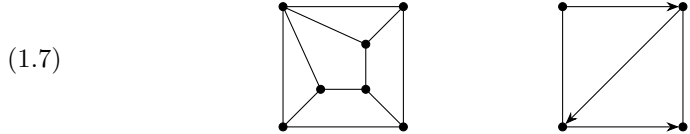
\begin{equation}
\begin{tikzpicture}[baseline]
\tikzset{
vertex/.style={circle,
inner sep=0pt, outer sep=0pt,
minimum width=3pt, 
fill=black, 
draw =black}}
\node[vertex] (a0) at (-0.8,0.8) {};
\node[vertex] (a1) at (0.8,0.8) {};
\node[vertex] (a2) at (0.8,-0.8) {};
\node[vertex] (a3) at (-0.8,-0.8) {};
\node[vertex] (b1) at (0.3,0.3) {};
\node[vertex] (b2) at (0.3,-0.3) {};
\node[vertex] (b3) at (-0.3,-0.3) {};
\draw[black] 
(a0) -- (a1) -- (a2) -- (a3) -- (a0)
(a0) -- (b1) -- (b2) -- (b3) -- (a0)
(a1) -- (b1)
(a2) -- (b2)
(a3) -- (b3)
;
\end{tikzpicture}
\hspace{2cm}
\begin{tikzpicture}[baseline]
\tikzset{
vertex/.style={circle,
inner sep=0pt, outer sep=0pt,
minimum width=3pt, 
fill=black, 
draw =black}} 
\node[vertex] (a0) at (-0.8,0.8) {};
\node[vertex] (a1) at (0.8,0.8) {};
\node[vertex] (a2) at (0.8,-0.8) {};
\node[vertex] (a3) at (-0.8,-0.8) {};
\draw[-{Stealth}] (a0) -- (a1);
\draw[-{Stealth}] (a1) -- (a3);
\draw[-{Stealth}] (a3) -- (a2);
\draw 
(a1) -- (a2)
(a0) -- (a3)
;
\end{tikzpicture}
\end{equation}
satisfy $(\mathcal V_2)$, but they are not diagonal, because $\MH_{3,4}(G,\mathbb{Z})\neq 0$
(see computations in Examples \ref{example:computation1} and \ref{example:computation2}). Below, we construct a more conceptual example of a non-diagonal digraph satisfying $(\mathcal V_2)$, which is defined by a combinatorial triangulation of $S^1\times I$.

Di and Zhang, together with the authors, studied a filtered simplicial set $\NN^*(G)$, called the filtered nerve of $G$, in \cite[\S 1]{di2024path} (see also \cite[\S 3]{hepworth2023reachability}). They considered the $\ell$-fundamental groupoid of a digraph $G$, defined as the fundamental groupoid of the simplicial set $\NN^\ell(G)$. Similarly, we define the $\ell$-fundamental category $\tau^\ell(G)$ of a digraph $G$ as the fundamental category of the simplicial set $\NN^\ell(G)$. We prove that for a digraph $G$ and an integer $\ell \geq 2$,
\begin{equation}
\MH_{2,k}(G) = 0 \text{ for } k > \ell
\hspace{5mm}
\Rightarrow
\hspace{5mm}
\tau^\ell(G) \text{ is a thin category.}
\end{equation}
As a corollary, we obtain that the GLMY-fundamental group of a diagonal undirected graph is trivial (Corollary~\ref{cor:fundamental_diagonal}).

It is known \cite[Corollary~4.5]{grigor2018fundamental} that $\pi^{\sf GLMY}_1(G)$ is isomorphic to $\pi_1({\sf CW}^2(G))$, where ${\sf CW}^2(G)$ is the two-dimensional CW-complex obtained from the geometric realization of $G$ by attaching 2-cells to all its squares and triangles. Therefore, we conclude that
\begin{equation}
G \text{ is a diagonal undirected graph}
\hspace{5mm}
\Rightarrow
\hspace{5mm}
{\sf CW}^2(G) \text{ is simply connected.}
\end{equation}

It is worth noting that the path homology of a diagonal undirected graph is also trivial \cite[Prop.~8.6]{asao2023magnitude}, and the bigraded path homology is also trivial \cite{hepworth2024bigraded}. Thus, diagonal undirected graphs, in some sense, behave like contractible spaces from the viewpoint of GLMY theory.

\subsection{Diagonal digraphs and Koszul algebras}

For a digraph $G$, Grigor'yan--Lin--Muranov--Yau consider a dg-algebra $\Omega^\bullet(G) = \Omega^\bullet(G, \KK)$, whose cohomology is the path cohomology ${\rm PH}^n(G) = H^n(\Omega^\bullet(G))$, and whose elements are called ``$d$-invariant forms'' \cite[\S 3.4]{grigor2012homologies}. The algebra $\Omega^\bullet(G)$ will be referred to as the path cochain algebra. We prove that $\Omega^\bullet(G)$ is isomorphic to a quotient of the path algebra:
\begin{equation}
\Omega^\bullet(G) \cong \KK G / T,
\end{equation}
where $T$ is an ideal generated by quadratic relations $t_{x,y}$ indexed by pairs of vertices at distance two (Theorem~\ref{theorem:omega}). By a result of Hepworth \cite[Th.~6.2]{hepworth2022magnitude}, the latter quotient algebra is isomorphic to the diagonal part of the magnitude cohomology algebra:
\begin{equation}
\Omega^\bullet(G) \cong \MH^{\sf diag}(G).
\end{equation}
This is a dual version of the result of Asao \cite[Lemma 6.8]{asao2023magnitude}.

We further show that, if $G$ satisfies $(\mathcal{V}_2)$ and $\KK$ is a field, then the distance algebra $\sigma G$ is quadratic and quadratic dual to the path cochain algebra:
\begin{equation}
(\sigma G)^! \cong \Omega^\bullet(G^{\text{op}}).
\end{equation}
Using this, we obtain a characterization of diagonal digraphs in terms of Koszul algebras (Theorem~\ref{th:Koszul}). Namely, we prove that the following statements about a finite digraph $G$ are equivalent:
\begin{enumerate}
\item $G$ is diagonal;
\item $\sigma G$ is Koszul for any field $\KK$;
\item $G$ satisfies $(\mathcal{V}_2)$ and $\Omega^\bullet(G)$ is Koszul for any field $\KK$.
\end{enumerate}

For a digraph $G$ satisfying $(\mathcal{V}_2)$, we also give an explicit description of the Koszul complex of the quadratic algebra $\Omega^\bullet(G)$ (Proposition~\ref{prop:koszul_complex}). This provides another equivalent description of diagonal digraphs from a completely different viewpoint (Corollary \ref{cor:diagonality:koszul_complex}).

\subsection{Extended Hasse diagrams}

Here we develop the ideas of Kaneta--Yoshinaga \cite[\S 5.3]{kaneta2021magnitude} in the setting of digraphs. For a pure simplicial complex $K$, we consider a digraph $\hat G_K$, whose vertices are the simplices of $K$ together with two additional vertices $\hat 0$ and $\hat 1$. There are three types of arrows in $\hat G_K$: pairs of simplices $(\sigma, \tau)$, where $\sigma$ is a face of $\tau$ such that $\dim(\tau) = \dim(\sigma) + 1$; pairs of the form $(\hat 0, \sigma)$, where $\sigma$ is a $0$-simplex; and pairs of the form $(\sigma, \hat 1)$, where $\sigma$ is a maximal simplex. The digraph $\hat G_K$ will be referred to as the extended Hasse diagram of $K$.

We describe the non-diagonal part of the magnitude homology $\MH_{n,\ell}(\hat G_K)$ for a triangulation $K$ of a topological manifold (Theorem~\ref{th:manifold}). Namely, we prove that, for a triangulation $K$ of a topological manifold with boundary $M$ and $n \neq \ell$, we have
\begin{equation}
\MH_{n,\ell}(\hat G_K) \cong
\begin{cases}
\bar H_{n-2}(M), & 
\ell - 2 = \dim(M), \\ 
0, & \text{otherwise.}
\end{cases}
\end{equation}

As a corollary, we obtain that for a triangulation $K$ of a closed manifold $M$ of dimension at least one, the digraph $\hat G_K$ is diagonal if and only if $M$ is a homology sphere. Therefore, the Koszul property for $\hat G_K$ is a topological invariant of $M$. Note that the Koszul property has already appeared in the literature as a topological invariant in a similar context \cite{sadofsky2011koszul}.

Since any smooth manifold with boundary has a triangulation, this theorem allows us to construct interesting examples of (non-)diagonal digraphs. For example, take $2 \leq n_0 < \ell_0$ and consider a triangulation $K$ of the product $M = S^{n_0 - 2} \times I^{\ell_0 - n_0}$. Then the non-diagonal part of the magnitude homology groups of $\hat G_K$ is concentrated in degree $(n_0, \ell_0)$.

\subsection{Acknowledgements}
The authors are grateful to Alexandra Zvonareva, Xin Fu and Semen Podkorytov for useful discussions.

\section{A general version of the Gruenberg formula}

In this section we present some background related to quasi-free algebras and prove a general version of the Gruenberg formula for a quotient algebra of a quasi-free algebra $A=F/I.$ In the following sections we use this formula for the path algebra associated with a digraph $F=\KK G.$

\subsection{Quasi-free algebras}
We fix a commutative ring $\KK$ and denote  by $\GG$ an abelian group. For future references, we will work in the setting of $\GG$-graded algebras over $\KK$ and $\GG$-graded modules over them in this section, but in the next sections we will only use $\GG=\ZZ.$ In this section the term `graded' means `$\GG$-graded'. All tensor products and hom-sets in this subsection will be $\GG$-graded. Some background about $\GG$-graded algebras and modules can be found in \cite[Appendix]{asao2024magnitude}.  For a graded algebra $A=\bigoplus_{g\in \GG} A_g$, we consider the graded algebra $A^e=A^{op}\otimes A.$ Then any graded $A$-bimodule can be identified with a right graded $A^e$-module. 

A graded algebra $F$ is called \emph{quasi-free} if the projective dimension of $F$ in the category of graded $F$-bimodules is at most one. If we set 
$\Omega^1_F=\Ker(\mu:F\otimes F \to F),$ where $\mu(a\otimes b)=ab,$ then we obtain a short exact sequence of graded $F$-bimodules 
\begin{equation}
 0 \longrightarrow \Omega^1_F \longrightarrow F\otimes F \longrightarrow F \longrightarrow 0,
\end{equation}
where $F\otimes F$ is a free graded bimodule of rank one. Therefore $F$ is quasi-free if and only if $\Omega^1_F$ is a projective graded $F$-bimodule (see \cite{cuntz1995algebra}, \cite{kontsevich2000noncommutative}, \cite{bondal2015coherence}). 

We also prove that under some conditions  right homogeneous ideals of a quasi-free graded algebra $F$ are projective graded $F$-modules.

\begin{lemma}\label{lemma:quasi-free-proj-dim}
Let $F$ be a quasi-free graded algebra and $M$ be a right graded $F$-module, which is projective as a $\KK$-module. Then the projective dimension of $M$ in the category of graded $F$-modules is at most one.  
\end{lemma}
\begin{proof}
It is well known that for any three rings $A,B,C$ and any $(A,B)$-bimodule $X$ there is a pair of adjoint functors
\begin{equation}
X\otimes_B - : {\rm Bimod}(B,C) \leftrightarrows {\rm Bimod}(A,C) : \Hom_A(X,-).
\end{equation}
In other words, there is a natural isomorphism (see \cite[\S III.11.4]{faith2012algebra}) 
\begin{equation}
\Hom_{(A,C)}(X\otimes_B Y,Z) \cong  \Hom_{(B,C)}(Y,\Hom_A(X,Z)).
\end{equation}
If we take $X=M,$ $Y=P,$ $A=\KK,$ $B=C=F,$ we obtain an isomorphism
\begin{equation}
    \Hom_F(M\otimes_F P, -  ) \cong \Hom_{F^e}(P,\Hom_\KK(M,-)).
\end{equation}
Since $M$ is projective over $\KK$ and $ P$ is projective over $F^e,$ we obtain that the functor $\Hom_{F^e}(P,\Hom_\KK(M,-))$ is exact. Therefore  $M\otimes_F P$ is a projective right graded $F$-module. Now consider a projective resolution in the category of graded bimodules $P_1\mono P_0 \epi F.$ Since $F$ is projective as a left $F$-module, the sequence splits as a short exact sequence of left graded modules. Therefore, tensoring by $M$ over $F$ we obtain a short exact sequence  $M\otimes_F P_1 \mono M\otimes_F P_0 \epi M.$ Therefore, $M\otimes_F P_1 \mono M\otimes_FP_0$ is a projective resolution of $M$ of length $1.$
\end{proof}

\begin{lemma}\label{lemma:proj_ideal}
Let $F$ be a quasi-free graded algebra and $I$ be a right homogeneous ideal of $F$ such that  $F/I$ is projective over $\KK.$  Then $I$ is a projective graded right $F$-module.
\end{lemma}
\begin{proof}
By Lemma \ref{lemma:quasi-free-proj-dim} we know that $F/I$ has projective dimension over $F$ at most one. Therefore, the short exact sequence $I\mono F\epi F/I$ implies that the projective dimension of $I$ over $F$ is  $0.$ 
\end{proof}

\begin{lemma}\label{lemma:projective_F/IJ}
Let $F$ be a quasi-free graded algebra, $J$ be a right homogeneous ideal of $F,$ and $I$ be a two-sided homogeneous ideal of $F.$ Assume that  $F/I$ and $ F/J$ are projective over $\KK.$ Then $F/JI$ is also projective over $\KK.$
\end{lemma}
\begin{proof}
By Lemma \ref{lemma:proj_ideal} $J$ is projective as a right $F$-module. Then $J/JI \cong J\otimes_F F/I$ is projective as a right $F/I$-module. Since $F/I$ is projective over $\KK,$ we obtain that $J/JI$ is also projective over $\KK.$ Then the short exact sequence $J/JI\mono F/JI \epi F/J$ implies that $F/JI$ is also projective over $\KK.$
\end{proof}

\begin{proposition}\label{prop:projective_products}
Let $F$ be a quasi-free graded algebra, $J$ be a right homogeneous ideal of $F,$ and $I_1,\dots,I_n$ be two-sided homogeneous ideals of $F.$ Assume that $F/J$ and $F/I_1,\dots, F/I_n$   are projective over $\KK.$  Then $F/(JI_1\dots I_n)$ is also projective over $\KK$ and $JI_1\dots I_n$ is a projective graded right $F$-module. 
\end{proposition}
\begin{proof}
It follows by induction from Lemma \ref{lemma:projective_F/IJ} and Lemma \ref{lemma:proj_ideal}.
\end{proof}

Next, we give two examples of quasi-free algebras:
the path algebra $\KK Q$ of a quiver $Q$ and the group algebra $\KK[\mathcal F]$ of a free group $\mathcal F.$

\begin{proposition}\label{prop:path_algebra_quasifree}  Let $Q=(Q_0,Q_1,s,t)$ be a quiver with a  finite set of vertices $Q_0$ equipped with a function $|\cdot|:Q_1\to \GG.$ Then the graded path algebra $\KK Q$ is quasi-free, where the grading is defined so that the degree of a path $\alpha_1\dots \alpha_n$ is the sum $|\alpha_1|+\dots +|\alpha_n|$ and the degree of an idempotent $e_x,x\in Q_0$ is zero.  
\end{proposition}
\begin{proof} Set $F=\KK Q$ and consider the following sequence of bimodules
\begin{equation}\label{eq:ses_KQ} 
0 \longrightarrow
\bigoplus_{\alpha\in Q_1} Fe_{s(\alpha)} \otimes \KK \alpha \otimes  e_{t(\alpha)} F
 \overset{\partial}\longrightarrow  \bigoplus_{x\in Q_0} Fe_x\otimes e_x F \overset{\mu} \longrightarrow F \longrightarrow 0,
\end{equation}
where $\mu(a\otimes b) = ab$ and $\partial(a\otimes \alpha \otimes  b) = a\alpha\otimes b - a\otimes \alpha b.$ We claim that it is a projective resolution of $F$ in the category of bimodules. Since $Fe_x\otimes e_y F$ is a direct summand of $F\otimes F,$ it is a projective $F$-bimodule. So we only need to prove that \eqref{eq:ses_KQ} is a short exact sequence. In order to do this we construct homomorphisms of right $F$-modules  $\sigma:F\to \bigoplus_{x\in Q_0} Fe_x\otimes e_x F$ and $\tau:\bigoplus_{x\in Q_0} Fe_x\otimes e_x F \to \bigoplus_{\alpha\in Q_1} Fe_{s(\alpha)} \otimes \KK \alpha \otimes  e_{t(\alpha)} F$ such that \begin{equation}\label{eq:direct_sum}
 \sigma \mu + \partial \tau={\rm id}, \hspace{5mm}  \mu \sigma = {\rm id}, \hspace{5mm} \tau \partial ={\rm id}. 
\end{equation} 
The morphisms are defined by formulas $\sigma(a)=e_x\otimes a, a\in e_xF$ and
\begin{equation}
\tau(\alpha_1\dots \alpha_n\otimes a) = \sum_{i=1}^n \alpha_1 \dots \alpha_{i-1}\otimes \alpha_i \otimes \alpha_{i+1} \dots \alpha_{n}a,    
\end{equation}
where $\alpha_i\in Q_1$ and $\alpha_1\dots\alpha_n$ is a path to a vertex $x$ and $a\in e_x F.$ Now the equations \eqref{eq:direct_sum} can be verified by a direct computation.
\end{proof}

\begin{corollary}
A free graded algebra $\KK\langle X \rangle$ with a grading defined by any map $|\cdot| : X\to \GG$ is quasi-free.
\end{corollary}
\begin{proof}
A free algebra can be presented as a path algebra of a quiver with one vertex.
\end{proof}

\begin{proposition}
Let $\FF=\FF(X)$ be a free group generated by a set $X.$ Then the (non-graded) group algebra $\KK[\FF]$ is quasi-free. 
\end{proposition}
\begin{proof} Set $F=\KK[\mathcal F]$ and
consider the following sequence of $F$-bimodules 
\begin{equation}\label{eq:free-resol}
0\longrightarrow \bigoplus_{x\in X} F\otimes \KK x\otimes F \overset{\partial}\longrightarrow  F \otimes  F \overset{\mu}\longrightarrow F \longrightarrow 0, 
\end{equation}
where $\mu(a\otimes b)=ab$ and $\partial(a\otimes x\otimes b) = ax\otimes b - a\otimes xb.$ It is easy to see that $F\otimes F$ and $F\otimes \KK x \otimes F$ are free bimodules. So we only need to show that the sequence \ref{eq:free-resol} is exact. In order to do this we construct homomorphisms of right $F$-modules  $\sigma:F\to F\otimes F$ and $\tau: F\otimes  F \to \bigoplus_{x\in X} F \otimes \KK x \otimes F$ such that the equations \eqref{eq:direct_sum} are satisfied. The map $\sigma$ is defined by the formula $\sigma(a)=1\otimes a.$ In order to define $\tau$ we first define it for an element of the form $x^\varepsilon\otimes 1,$ where $x\in X$ and $\varepsilon\in \{-1,1\}.$ It is defined by the formula
\begin{equation}
\tau(x^\varepsilon\otimes 1) = \begin{cases}
1\otimes x\otimes 1, & \varepsilon=1\\
- x^{-1} \otimes x \otimes x^{-1},& \varepsilon=-1. 
\end{cases}
\end{equation}
note that we have 
\begin{equation}\label{eq:partial_tau}
    \partial(\tau(x^\varepsilon \otimes 1)) = x^\varepsilon \otimes 1 - 1 \otimes x^\varepsilon.
\end{equation}
In general, $\tau$ is defined so that 
\begin{equation}
\tau(x_1^{\varepsilon_1} \dots x_n^{\varepsilon_n} \otimes a ) = \sum_{i=1}^n x_1^{\varepsilon_1} \dots x_{i-1}^{\varepsilon_{i-1}} \tau(x_i^{\varepsilon_i}\otimes 1) x_{i+1}^{\varepsilon_{i+1}} \dots x_n^{\varepsilon_n} a.
\end{equation}
Now, using the formula \eqref{eq:partial_tau} and the fact that $\partial$ is a bimodule homomorphism, it is easy to check the equations \eqref{eq:direct_sum}.
\end{proof}

\subsection{The Gruenberg formula}

The following theorem is a generalization of the Gruenberg resolution  \cite{gruenberg1960resolutions}.

\begin{theorem}[The Gruenberg resolution]
Let $F$ be a quasi-free graded algebra over a commutative ring $\KK$, $I$ be a  two-sided homogeneous ideal of $F$ and $J$ be a right homogeneous ideal of $F$ such that $I\subseteq J.$ Assume that $F/I$ and $F/J$ are projective $\KK$-modules and set  $A=F/I$ and $S=F/J.$ Then there is a projective resolution $P_\bullet$ of the graded right $A$-module $S$  
\begin{equation}
\dots \to  \frac{JI}{JI^2} \to  \frac{I}{I^2}  \to  \frac{J}{JI} \to  \frac{F}{I} \to  S
\end{equation}
such that
\begin{equation}
P_{2n} = \frac{I^n}{I^{n+1}}, \hspace{1cm} P_{2n+1} = \frac{JI^n}{JI^{n+1}}
\end{equation}
and the differential and the augmentation are induced by the embeddings $JI^n \hookrightarrow I^n \hookrightarrow JI^{n-1}.$
\end{theorem}
\begin{proof}
The fact that the chain complex $\dots \to  P_1\to P_0\to S\to 0$ is exact is obvious. Proposition \ref{prop:projective_products} implies that $I^n,JI^n$ are projective $F$-modules. Therefore $I^n/I^{n+1}, JI^n/JI^{n+1}$ are projective $A$-modules.  
\end{proof}

If $A$ is a graded algebra, and $M,M'$ are a right and left graded $A$-modules, then we can consider a bigraded Tor functor
\begin{equation}
\Tor^A_{*,*}(M,M')    
\end{equation}
which is defined as $\Tor^A_{n,g}(M,M')=(H_n(P_\bullet\otimes_A M'))_g$ for $n\in \ZZ$ and $g\in \GG,$ where $P_\bullet$ is a graded projective resolution of $M$ and $\otimes_A$ denoted the graded tensor product over $A$ (see \cite[\S 7.2]{asao2024magnitude}). 

The following theorem is a generalization of the classical Gruenberg formula \cite[\S 3.7]{gruenberg2006cohomological} and a generalization of a similar statement of Govorov in the case of graded algebras \cite[Lemma 1]{govorov1973dimension}.

\begin{theorem}[The Gruenberg formula] \label{theorem:Gruenberg_formulas}
Let $F$ be a quasi-free graded algebra over a commutative ring $\KK$, $I$ be a two-sided homogeneous ideal of $F$, $J$ be a right homogeneous ideal of $F,$ and $J'$ be a left homogeneous ideal of $F$ such that $I\subseteq J$ and $I\subseteq J'.$ Assume that $F/I$ and $F/J$ are projective $\KK$-modules and set  
\begin{equation}
A=F/I, \hspace{5mm} S=F/J, \hspace{5mm}  S'=F/J'.    
\end{equation}
Then for $n\geq 0$ there are isomorphisms of graded modules 
\begin{equation}\label{eq:gruenberg_formulas}
\Tor_{2n,*}^A(S,S')\cong \frac{I^n \cap JI^{n-1}J'}{JI^n + I^nJ'}, \hspace{1cm} \Tor^A_{2n+1,*}(S,S') \cong \frac{JI^n \cap I^n J'}{ I^{n+1}+JI^nJ'}. 
\end{equation}
\end{theorem}
\begin{proof}
For any right graded $A$-module $M$ we have $M\otimes_A S'=M/MJ'.$ Let $P_\bullet$ be the Gruenberg resolution whose components are  $P_{2n}=I^n/I^{n+1}$ and $P_{2n+1}=JI^n/JI^{n+1}.$ Then $P_{2n}\otimes_A S'\cong I^n/I^nJ'$ and $P_{2n+1}\otimes_A S'\cong JI^n/JI^nJ'.$ Therefore we have 
\begin{equation}
\begin{split}
\Ker(P_{2n}\otimes_A S' \to P_{2n-1}\otimes_AS' )&\cong (I^n \cap JI^nJ')/I^nJ',\\
\Im(P_{2n+1}\otimes_A S' \to P_{2n}\otimes_A S' ) & \cong (JI^n+I^nJ')/I^nJ',\\
\Ker(P_{2n+1}\otimes_A S' \to P_{2n}\otimes_A S') &\cong (JI^n\cap I^nJ')/JI^nJ',\\
\Im(P_{2n+2}\otimes_A S'\to P_{2n+1}\otimes_A S' )&\cong (I^{n+1}+JI^nJ')/JI^nJ'.
\end{split}
\end{equation}
This proves the claim. 
\end{proof}

\begin{remark}[Naturality of the Gruenberg formulas]\label{remark:Naturalness_of_Gruenberg_formulas}
The Gruenberg formulas from Theorem \ref{theorem:Gruenberg_formulas} are natural in $S$ and $S'$ in the following sense. Assume we have another two ideals $\tilde J, \tilde J'$ satisfying the same assumptions as $J,J'$  such that $J\subseteq \tilde  J$ and $J'\subseteq \tilde J'.$ If we set $\tilde S=F/\tilde J$ and $\tilde S'=F/\tilde J',$ we have epimorphisms of $A$-modules $S\epi \tilde S$ and $S'\epi \tilde S',$ that induce a morphism
\begin{equation}\label{eq:tor_naturallnes}
 \Tor^A_{*,*}(S,S') \to \Tor^A_{*,*}(\tilde S, \tilde S').   
\end{equation}
On the other hand, we have morphisms 
\begin{equation}\label{eq:gruenberg_naturallnes}
  \frac{I^n \cap JI^{n-1}J'}{JI^n + I^nJ'} \longrightarrow \frac{I^n \cap \tilde JI^{n-1}\tilde J'}{\tilde JI^n + I^n\tilde J'}, \hspace{1cm}  \frac{JI^n \cap I^n J'}{ I^{n+1}+JI^nJ'} \to \frac{\tilde JI^n \cap I^n \tilde J'}{ I^{n+1}+ \tilde JI^n\tilde J'} 
\end{equation}
induced by inclusions. Then the isomorphisms \eqref{eq:gruenberg_formulas} are consistent with morphisms \eqref{eq:tor_naturallnes} and \eqref{eq:gruenberg_naturallnes}, in the sense that the corresponding diagrams are commutative.
\end{remark}

As a corollary we obtain the classical result of Gruenberg.

\begin{corollary}
Let $\mathcal F$ be a free group and $\mathcal G=\mathcal F/\mathcal R$ be its quotient group. Denote by $\bf f$ the augmentation ideal of the group algebra $\KK[\mathcal F]$ and set ${\bf r}=\Ker(\KK[\mathcal F]\to \KK[\mathcal G] ).$ Then there are isomorphisms
\begin{equation}
H_{2n}(\mathcal G,\KK)\cong \frac{{\bf r}^n \cap {\bf f}{\bf r}^{n-1}{\bf f}}{{\bf f}{\bf r}^n + {\bf r}^n{\bf f}}, \hspace{1cm} H_{2n+1}(\mathcal G,\KK) \cong \frac{{\bf f}{\bf r}^n \cap {\bf r}^n {\bf f}}{ {\bf r}^{n+1}+{\bf f}{\bf r}^n{\bf f}}. 
\end{equation}
\end{corollary}

\section{The Gruenberg formula for the magnitude homology} 

In this section we recall the basic information about the magnitude homology groups of digraphs $\MH_{n,\ell}(G)$ and about the decomposition into a direct sum of directed components  $\MH_{n,\ell}(G) = \bigoplus_{x,y} \MH_{n,\ell}(G)_{x,y}$. We will recall the theorem about the presentation of the magnitude homology groups as a Tor functor
$\MH_{n,\ell}(G) = \Tor^{\sigma G}_{n,\ell}(S,S)$
and prove a version of it for the directed components 
\begin{equation}
\MH_{n,\ell}(G)_{x,y} = \Tor^{\sigma G}_{n,\ell}(S_x,S_y).
\end{equation}
Using this we will prove the Gruenberg formula for the magnitude homology groups (Theorem \ref{theorem:Gruenberg_for_magnitude}).
We will also demonstrate the computational convenience of this formula with examples  (Example \ref{example:computation1}, Example \ref{example:computation2}).

\subsection{Magnitude homology} Let $G$ be a digraph. For any vertices $x,y$ of $G$ we denote by $d(x,y)$ the infimum of lengths of paths from $x$ to $y.$ For a tuple of vertices $(x_0,\dots,x_n)$ we say that $n$ is the length of the tuple and 
\begin{equation}
|x_0,\dots,x_n| = \sum_{i=1}^{n-1} d(x_i,x_{i+1})
\end{equation}
is the norm of the tuple. For any nonnegative integer $\ell$ the chain complex of $\KK$-modules  $\MC_{\bullet,\ell}(G)$ is defined so that 
\begin{equation}
\MC_{n,\ell}(G) = \KK\cdot \{(x_0,\dots,x_n)\mid |x_0,\dots,x_n| = \ell, x_i\ne x_{i+1} \}
\end{equation}
and the differential is defined by the formula $\partial_n=\sum_{i=0}^{n}(-1)^i \partial_{n,i},$ where 
\begin{equation}\label{eq:d_ni}
\partial_{n,i}(x_0,\dots,x_n)=
\begin{cases}
(x_0,\dots,\hat x_i,\dots,x_n), & |x_0,\dots, \hat x_i, \dots ,x_n|=\ell \\
0, & \text{otherwise.}
\end{cases}
\end{equation}
Note that $\partial_{n,0}=\partial_{n,n}=0.$ The magnitude homology with coefficients in a commutative ring $\KK$ is defined as the homology of this complex $\MH_{n,\ell}(G)=H_n( \MC_{\bullet,\ell}(G)).$ If we need to specify the commutative ring $\KK,$ we use the notations $\MC_{\bullet,\ell}(G,\KK)$ and $\MH_{n,\ell}(G,\KK).$

Following \cite[Prop.2.9]{asao2021geometric} \cite[Th.5.11]{kaneta2021magnitude},  \cite{asao2024girth}, \cite[Lemma 9.10]{hepworth2024bigraded} for any two vertices $x,y$ of $G$ we consider a free submodule $\MC_{n,\ell}(x,y)\subseteq \MC_{n,\ell}(G)$ generated by tuples of the form $(x,x_1,\dots,x_{n-1},y).$ It is easy to see that these submodules form a chain subcomplex $\MC_{\bullet,\ell}(x,y)\subseteq \MC_{\bullet,\ell}(G)$ and there is a decomposition 
\begin{equation}
\MC_{\bullet,\ell}(G)= \bigoplus_{x,y} \MC_{\bullet,\ell}(G)_{x,y}.
\end{equation}
This decomposition induces a decomposition of homology groups
\begin{equation}\label{eq:direct_decomposition}
\MH_{n,\ell}(G) \cong \bigoplus_{x,y} \MH_{n,\ell}(G)_{x,y}.
\end{equation}
The group $\MH_{n,\ell}(G)_{x,y}$ will be referred to as the directed component of the magnitude homology (in the direction from $x$ to $y$). 

The magnitude cohomology is defined by the formula $\MH^{n,\ell}(G)=H^n( \MC^{\bullet,\ell}(G)),$ where
$\MC^{\bullet,\ell}(G) = \Hom_{\KK}( \MC_{\bullet,\ell}(G),\KK).$
So, we can also consider the complexes $\MC^{\bullet,\ell}(G)_{x,y}=\Hom_\KK(\MC_{\bullet,\ell}(G)_{x,y},\KK),$ define the directed magnitude cohomology  $\MH^{n,\ell}(G)_{x,y}$ as their cohomology and obtain 
\begin{equation} \MC^{\bullet,\ell}(G)\cong \prod_{x,y} \MC^{\bullet,\ell}(G)_{x,y}, \hspace{5mm}
 \MH^{n,\ell}(G) = \prod_{x,y} \MH^{n,\ell}(G)_{x,y}.   
\end{equation}

Recall \cite{leinster2019magnitude} that the magnitude ${\rm Mag}(G)$ of a finite digraph $G$ with $m$ vertices can be defined as an integral formal power series in $\ZZ[\![q]\!],$
\begin{equation}
    {\rm Mag}(G) = \sum_{x,y} (Z^{-1}_G)_{x,y},
\end{equation}
where $Z_G$ is an $m\times m$ matrix over $\ZZ[\![q]\!]$, indexed by vertices, defined by $(Z_G)_{x,y}=q^{d(x,y)},$ assuming $q^\infty=0.$ 
The motivation for introducing the magnitude homology in \cite{hepworth2017categorifying} was that for a field $\KK$ we have 
\begin{equation}
\Mag(G) = \sum_{n,\ell} (-1)^n \dim_\KK (\MH_{n,\ell}(G))\cdot q^\ell.
\end{equation}
Let us prove a component-wise version of this formula. 

\begin{proposition}\label{prop:matrix} Let $G$ be a finite digraph and $\KK$ be a field. Then for any $\ell\geq 0$ and vertices $x,y$ we have
\begin{equation}\label{eq:mafnitude_lxy}
(Z_G^{-1})_{x,y} = \sum_{n,\ell} (-1)^n \dim_\KK (\MH_{n,\ell}(G)_{x,y}) \cdot q^\ell.
\end{equation}
\end{proposition}
\begin{proof}
First we note that $\MC_{n,\ell}(G)=0$ for $n>\ell,$ and the sum in the formula \eqref{eq:mafnitude_lxy} is finite. Since $\sum_n (-1)^n \dim_\KK (\MC_{n,\ell}(G)_{x,y})=\sum_n (-1)^n \dim_\KK (\MH_{n,\ell}(G)_{x,y}),$ it is sufficient to prove that $\Mag_\ell(G)_{x,y} = \sum_n (-1)^n \dim_\KK (\MC_{n,\ell}(G)_{x,y}).$ The matrix $Z_G$ can be presented as $Z_G=I+\tilde Z_G,$ where $I$ is the identity matrix and $(\tilde Z_G)_{x,y}=q^{d(x,y)}$ for $x\neq y$ and $(\tilde Z_G)_{x,x}=0.$ 
Since all components of $\tilde Z_G$ are divisible by $q,$ the components of $\tilde Z_G^n$ are divisible by $q^n,$ and we obtain $Z_G^{-1}=\sum_n (-1)^n \tilde Z_G^n.$ It is easy to check by induction that $(\tilde Z_G^n)_{x,y}=\sum_{x=x_0\ne \dots \ne x_n=y} q^{|x_0,\dots,x_n|},$ where the sum runs over all tuples  $(x_0,\dots,x_n)$ such that $x_i\neq x_{i+1}$ for any $i$ and $x_0=x,x_n=y.$ Therefore, $(\tilde Z_G^n)_{x,y}=\sum_\ell  \dim_\KK (\MC_{n,\ell}(G)_{x,y}) \cdot q^\ell.$ The result follows. 
\end{proof}

\subsection{Non-normalized magnitude chain complex} 
For a general simplicial $\KK$-module $X$, one can associate two chain complexes. The non-normalized chain complex $C(X)$ whose components are $C_n(X)=X_n$ and the differentials are defined by $\partial_n=\sum_{i}(-1)^i \partial_{n,i},$ where $\partial_{n,i}:X_n\to X_{n-1}$ are the face maps. And the normalized chain complex $N(X),$ which is a quotient of $C(X)$ and it has components $N_n(X)=X_n/(\sum_i \sigma_{n-1,i}(X_{n-1})),$ where $\sigma_{n-1,i}:X_{n-1}\to X_n$ are the degeneracy maps. It is well known that the map $C(X) \epi N(X)$ is a chain homotopy equivalence. In particular $H_*(X):=H_*(C(X))\cong H_*(N(X)).$ 

The chain complex $\MC_{\bullet,\ell}(G)$ is a normalized chain complex of a simplicial $\KK$-module $\widetilde{\MC}_{\bullet,\ell}(G)$ freely generated by the magnitude simplicial set $\mathscr{M}^\ell(G)$ defined in \cite[(1.9)]{di2024path} (see also \cite[Definition 5.7]{leinster2021magnitude}).  The components of $\widetilde{\MC}_{\bullet,\ell}(G)$ are generated by tuples with possible consecutive repetitions 
\begin{equation}
\widetilde{\MC}_{\bullet,\ell}(G)=\KK\cdot \{(x_0,\dots,x_n)\mid |x_0,\dots,x_n|=\ell\}. 
\end{equation}
The face maps $\partial_{n,i}$ of $\widetilde{\MC}_{\bullet,\ell}(G)$ are defined by \eqref{eq:d_ni} and the degeneracy maps $\sigma_{n,i}$ are defined by $\sigma_{n,i}(x_0,\dots,x_n)=(x_0,\dots,x_i,x_i,\dots,x_n).$ In this case $\partial_{n,0}$ and $\partial_{n,n}$ are nontrivial maps: $\partial_{n,0}(x_0,\dots,x_n)$ is nontrivial if and only if $x_0=x_1,$ and $\partial_{n,n}(x_0,\dots,x_n)$ is nontrivial if and only if $x_{n-1}=x_n$. Therefore $\widetilde \MC_{\bullet,\ell}(G)$ is chain homotopy equivalent to $\MC_{\bullet,\ell}(G)$ and 
\begin{equation}
\MH_{n,\ell}(G) \cong H_n( \widetilde \MC_{\bullet,\ell}(G)).
\end{equation} 
Applying $\Hom_\KK(-,\KK)$ to chain homotopy equivalent complexes we obtain homotopy equivalent complexes. Therefore
\begin{equation}
\MH^{n,\ell}(G) \cong H^n(  \widetilde{\MC}^{\bullet,\ell}(G)),
\end{equation}
where $\widetilde{\MC}^{\bullet,\ell}(G)=\Hom_\KK(\widetilde{\MC}_{\bullet,\ell}(G),\KK).$

Note that we can similarly define simplicial submodules
\begin{equation}
\widetilde{\MC}_{\bullet,\ell}(G)_{x,y} \subseteq \widetilde{\MC}_{\bullet,\ell}(G)
\end{equation}
and obtain that the chain complex $\widetilde{\MC}_{\bullet,\ell}(G)_{x,y}$ is homotopy equivalent to $\MC_{\bullet,\ell}(G)_{x,y}$ and 
\begin{equation}
\MH_{n,\ell}(G)_{x,y} \cong   H_n(\widetilde{\MC}_{\bullet,\ell}(G)_{x,y}).\end{equation}
We also have 
\begin{equation}
\MH^{n,\ell}(G)_{x,y} \cong  H^n( \widetilde \MC^{\bullet,\ell}(G)_{x,y}),
\end{equation}
where $\widetilde \MC^{\bullet,\ell}(G)_{x,y}\cong \Hom_\KK( \widetilde \MC_{\bullet,\ell}(G)_{x,y},\KK).$

\subsection{Magnitude homology is a  derived functor}

For a finite digraph $G$ we consider the graded path algebra $\KK G.$ A path in $G$ is a tuple of vertices $p=(p_0,\dots,p_n),$ such that $(p_i,p_{i+1})$ is an arrow in $G$ and $n\geq 0.$ The number $n$ is called the length of the path. In particular, we consider paths of length zero $(x),$ which we also denote by $e_x=(x).$ Then $\KK G$ is freely spanned over $\KK$ by paths. The multiplication is defined by the concatenation, if it is defined, and zero otherwise. The element $1:=\sum_{x\in X} e_x $ is the unit of the algebra. The grading is defined so that $(\KK G)_n$ is spanned by paths of length $n.$ A path $(p_0,\dots,p_n)$ is called a shortest path if $d(p_0,p_n)=n,$ and long path otherwise. Consider an ideal $R$ of $\KK G$ defined by the following relations
\begin{itemize}
\item[(R1)] $p=q,$ if $p$ and $q$ are two different shortest paths connecting the same vertices.
\item[(R2)] $p=0,$ if $p$ is a long path.
\end{itemize}
A long path $p=(p_0,\dots,p_n)$ is called a minimal long path if $(p_0,\dots,p_{n-1})$ and $(p_1,\dots,p_n)$ are shortest paths. Then (R2) can be replaced by 
\begin{itemize}
    \item[(R2')] $p=0,$ if $p$ is a minimal long path.
\end{itemize}

We denote by $J$ the ideal of $\KK G$ generated by all paths of length $\geq 1$ and set
$S=\KK G/J.$

The distance algebra $\sigma G$ is freely spanned by pairs $(x,y)$ of vertices of $G$ such that there is a path from $x$ to $y.$ The product is defined so that 
\begin{equation}
    (x,y)(y',z) = 
    \begin{cases}
       (x,z), & y=y', \  d(x,y)+d(y,z)=d(x,z)\\
       0,& \text{otherwise.}
    \end{cases}
\end{equation}
The grading is defined by $|(x,y)|=d(x,y).$ If we need to specify $\KK$ we denote $\sigma G$ by $\sigma_\KK G.$
The map $\KK G\to \sigma G$ that sends a shortest path $(x_0,\dots,x_n)$ to $(x_0,x_n),$ and a long path to zero, defines an isomorphism \cite[Prop.6.1]{asao2024magnitude}
\begin{equation}
\KK G/R \cong \sigma G.
\end{equation}

\begin{theorem}[{\cite[Th.6.2]{asao2024magnitude}}]\label{theorem:magnitude_as_derived}
Let $\KK$ be a commutative ring and $G$ be a finite digraph. Then there are isomorphisms
\begin{equation}\label{eq:derived_functors}
\MH_{n,\ell}(G)\cong {\rm Tor}^{\KK G/R}_{n,\ell} (S,S), \hspace{1cm} \MH^{n,\ell}(G)\cong \Ext^{n,\ell}_{\KK G/R}(S,S). 
\end{equation}
\end{theorem}

Let us prove a version of this theorem for the directed components of the magnitude homology. Note that $S$ can be decomposed into the direct sum of modules 
\begin{equation}
S=\bigoplus_{x\in G_0} S_x,
\end{equation}
where $S_x=\KK e_x.$ We can present $S_x$ as
$S_x = \KK G/J_x,$
where $J_x$ is the ideal generated by all paths except $e_x.$ In other words 
$J_x=J+\sum_{y\ne x} \KK e_y.$

Assume that $\varphi:V\overset{\cong}\to U$ are isomorphic modules and each of them is isomorphic to a direct sum indexed by the same set of indexes $V\cong \bigoplus_{a\in \mathcal A} V_a$ and $U\cong \bigoplus_{a\in \mathcal A} U_a,$ where $\varphi_{a,b} : V_a \overset{\cong}\to U_a.$ We say that the isomorphisms $\varphi_{a,b}$ are compatible with the isomorphism $\varphi$ if the diagram 
\begin{equation}
\begin{tikzcd}
V\ar[r] \ar[d,"\varphi"] & \bigoplus_{a} V_a \ar[d,"\bigoplus_{a} \varphi_a "] \\
U \ar[r] & \bigoplus_{a} U_a
\end{tikzcd}
\end{equation}
is commutative.

\begin{theorem}\label{theorem:directed_magnitude_defived}
Let $\KK$ be a commutative ring and $G$ be a finite digraph and $x,y$ be vertices of $G.$ Then there are isomorphisms
\begin{equation}
\MH_{n,\ell}(G)_{x,y}\cong \Tor^{\KK G/R}_{n,\ell} (S_x,S_y), \hspace{1cm} \MH^{n,\ell}(G)_{x,y}\cong \Ext^{n,\ell}_{\KK G/R}(S_x,S_y) 
\end{equation}
which are compatible with the isomorphisms \eqref{eq:derived_functors} with respect to the decompositions defined by the decomposition $S=\bigoplus_{x} S_x.$
\end{theorem}
\begin{remark}
In the statement of Theorem \ref{theorem:directed_magnitude_defived} in the formula for homology we treat $S_x$ as a right module and $S_y$ as a left module; but in the formula for cohomology $S_x$ and $S_y$ are both treated as right modules
\end{remark}
\begin{proof}[Proof of Theorem \ref{theorem:directed_magnitude_defived}]
In this proof, we will work over $\sigma G$ rather than over $\KK G/R.$ Then $S_x$ can be defined as $S_x = \sigma  G/J'_x,$ where $J'_x$ is generated by all pairs $(y,z)\ne (x,x).$ In \cite[Rem.5.4]{asao2024magnitude} the authors constructed a projective resolution $P_\bullet$ of the right graded module $S=\bigoplus_{x\in G_0} S_x$  such that 
\begin{equation}
(P_n)_\ell =\KK\cdot  
\{(x_0,\dots,x_n,y)\mid |x_0,\dots,x_n,y|=\ell \},
\end{equation}
and the differential is defined by $\partial_n=\sum_{i=0}^n (-1)\partial_{n,i},$ where $\partial_{n,i}$ is defined by 
\begin{equation}
\partial_{n,i}(x_0,\dots,x_n,y)= 
\begin{cases}
(x_0,\dots,\hat x_i,\dots,x_n,y), & |x_0,\dots,\hat x_i,\dots,x_n,y|=|x_0,\dots,x_n,y|\\
0, & \text{otherwise.}
\end{cases}
\end{equation}
The structure of a $\sigma G$-module on $P_n$ is defined by $(x_0,\dots,x_n,y)\cdot (y',z)=(x_0,\dots,x_n,z),$ if $|x_0,\dots,x_n,y| + |y,z| =|x_0,\dots,x_n,z|$ and $y=y',$ and by $(x_0,\dots,x_n,y)\cdot (y',z)=0$ otherwise.  The map $P_0\to S$ is defined by $(x_0,y)\mapsto (x_0),$ if $x_0=y,$ and $(x_0,y)\mapsto 0$ otherwise. For each $x\in G_0$ consider a submodule $P_n^x$ of $P_n$ spanned by tuples $(x_0,\dots,x_n,y)$ such that $x_0=x.$ Then 
\begin{equation}
P_\bullet = \bigoplus_{x\in G_0} P^x_\bullet
\end{equation}
and $P^x_\bullet$ is a projective resolution of $S_x.$ It is easy to see that the embeddings and projections $S_x\to S \to S_x$ can be lifted to the corresponding embeddings and projections for the resolutions $P^x_\bullet \to P_\bullet \to P^x_\bullet.$ For any graded $\sigma G$-module $M$ we have 
$M\otimes_{\sigma G} S_y=M/MJ'_y$ and $\Hom_{\sigma G}(M,S_y)=\Hom_\KK(M/MJ'_y,\KK).$ Using these formulas, it is easy to check that 
\begin{equation}
P^x_\bullet\otimes_{\sigma G} S_y \cong \widetilde{\MC}_{\bullet,*}(G)_{x,y}, \hspace{5mm} \Hom_{\sigma G}(P^x_{\bullet},S_y)\cong \widetilde{\MC}^{\bullet,*}(G)_{x,y}.
\end{equation}
Similarly we obtain the isomorphisms 
\begin{equation}
P_\bullet\otimes_{\sigma G} S \cong \widetilde{\MC}_{\bullet,*}(G), \hspace{5mm} \Hom_{\sigma G}(P_{\bullet},S)\cong \widetilde{\MC}^{\bullet,*}(G).
\end{equation}
Hence the theorem follows. 
\end{proof}

\subsection{The Gruenberg formula}

Note that $\KK G=\bigoplus_{x,y} e_x (\KK G)e_y.$ For any ideal $\mathfrak{a}$ of $\KK G,$ this decomposition induces a decomposition $\mathfrak{a} = \bigoplus_{u,v} e_v\mathfrak{a}e_u.$ Therefore, for any two ideals $\mathfrak{b}\subseteq \mathfrak{a} \subseteq \KK G$ we have 
\begin{equation}\label{eq:a/b}
   \frac{\mathfrak{a}}{\mathfrak{b}} \cong \bigoplus_{x,y}  \frac{e_x \mathfrak{a}e_y}{e_x \mathfrak{b}e_y}.
\end{equation}

\begin{theorem}\label{theorem:Gruenberg_for_magnitude}
For any finite digraph $G,$ any commutative ring $\KK$ and any $n\geq 0$ there are isomorphisms of graded $\KK$-modules
\begin{equation}\label{eq:gruenberg_magnitude}
\MH_{2n,*}(G)\cong \frac{R^n \cap JR^{n-1}J}{JR^n + R^nJ}, \hspace{1cm} \MH_{2n+1,*}(G) \cong \frac{JR^n \cap R^n J}{ R^{n+1}+JR^nJ}. 
\end{equation}
Moreover, for any vertices $x,y$ there are isomorphisms 
\begin{equation}\label{eq:Gruenberg_even_directed}
\MH_{2n,*}(G)_{x,y}\cong \frac{e_x(R^n \cap JR^{n-1}J)e_y}{e_x(JR^n + R^nJ)e_y}  
\end{equation}
and
\begin{equation} \label{eq:Gruenberg_odd_directed}
\MH_{2n+1,*}(G)_{x,y} \cong \frac{e_x(JR^n \cap R^n J)e_y}{e_x( R^{n+1}+JR^nJ)e_y}
\end{equation}
compatible with the isomorphisms \eqref{eq:gruenberg_magnitude}. 
\end{theorem}
\begin{proof}
The isomorphisms \eqref{eq:gruenberg_magnitude} follow from Theorem \ref{theorem:Gruenberg_formulas}, Proposition \ref{prop:path_algebra_quasifree}, Theorem \ref{theorem:magnitude_as_derived} and the fact that $\sigma G\cong \KK G/R$ and $S$ are free over $\KK.$ 

Let us prove the isomorphism   \eqref{eq:Gruenberg_even_directed} and note that the isomorphism \eqref{eq:Gruenberg_odd_directed} can be proved similarly. Theorem \ref{theorem:Gruenberg_formulas} and Theorem \ref{theorem:directed_magnitude_defived} imply the isomorphism 
\begin{equation}
\MH_{2n,*}(G)_{x,y}\cong \frac{R^n \cap J_xR^{n-1}J_y}{J_xR^n + R^nJ_y}. 
\end{equation}
Naturality of the Gruenberg formulas 
(Remark \ref{remark:Naturalness_of_Gruenberg_formulas}) and the fact that the isomorphisms of Theorem \ref{theorem:directed_magnitude_defived} are compatible with the isomorphism from Theorem \ref{theorem:magnitude_as_derived} imply that the projection $\MH_{2n,*}(G)\epi \MH_{2n,*}(G)_{x,y}$ is isomorphic to the map induced by embeddings
\begin{equation}\label{eq:map-gruenberg_even}
\frac{R^n \cap JR^{n-1}J}{JR^n + R^nJ} \longrightarrow   \frac{R^n \cap J_xR^{n-1}J_y}{J_xR^n + R^nJ_y}.
\end{equation}
The isomorphism \eqref{eq:a/b} gives a decomposition 
\begin{equation}
  \frac{R^n \cap JR^{n-1}J}{JR^n + R^nJ} \  \cong \  \bigoplus_{v,u} \frac{e_v(R^n \cap JR^{n-1}J)e_u}{e_v(JR^n + R^nJ)e_u}. 
\end{equation}
Therefore, in order to prove the isomorphism \eqref{eq:Gruenberg_even_directed}, it is sufficient to prove that the kernel of \eqref{eq:map-gruenberg_even} contains $(e_v(R^n \cap JR^{n-1}J)e_u)/(e_v(JR^n + R^nJ)e_u)$ for $(v,u)\neq (x,y),$ and intersects trivially with $(e_x(R^n \cap JR^{n-1}J)e_y)/(e_x(JR^n + R^nJ)e_y).$ Indeed, if $(v,u)\ne (x,y)$ and $a\in e_v(R^n \cap JR^{n-1}J)e_u,$ then $a\in e_v R^n e_u,$ and hence $a\in J_xR^n+R^n J_y.$ On the other hand we see that $e_xJ=e_xJ_x$ and $Je_y=J_ye_y.$ So, if $a\in e_x(R^n \cap JR^{n-1}J)e_y$ and $a\in J_xR^n+R^nJ_y,$ then $a=e_xae_y,$ and hence $a\in e_x(J_xR^n+R^nJ_y)e_y=e_x(JR^n+R^nJ)e_y.$
\end{proof}

\subsection{Examples of computations}

In this subsection for convenience of computations we simplify notations for the idempotents $x:=e_x.$

\begin{example}\label{example:computation1} In this example we use the Gruenberg formula to show that $\MH_{3,4}(G,\ZZ)$ is nontrivial for the following digraph.
\begin{equation}
\begin{tikzpicture}[baseline]
\tikzset{
vertex/.style={circle,
inner sep=0pt, outer sep=1pt,
minimum width=3pt, 
}} 
\node[vertex] (a0) at (-0.5,0.5) {$c_0$};
\node[vertex] (a1) at (0.5,0.5) {$a$};
\node[vertex] (a2) at (0.5,-0.5) {$c_1$};
\node[vertex] (a3) at (-0.5,-0.5) {$b$};
\draw[-{Stealth}] (a0) -- (a1);
\draw[-{Stealth}] (a1) -- (a3);
\draw[-{Stealth}] (a3) -- (a2);
\draw 
(a1) -- (a2)
(a0) -- (a3)
;
\end{tikzpicture}
\end{equation} 
Here an undirected edge means two arrows in opposite directions. There are two paths of length $4$ from $a$ to $b:$ $p_{i}=(a,b,c_i,a,b)$ for $i=0,1.$ We will prove that  
\begin{equation}
 \MH_{3,4}(G,\ZZ)_{a,b}\cong 
 \left( \frac{a(RJ\cap JR)b}{a(R^2+JRJ)b}\right)_4  = \frac{ \langle p_0,p_2\rangle}{ \langle p_0-p_2 \rangle }  \cong \ZZ.
\end{equation}
 Since $(a,b,c_i,a)$ and $(b,c_i,a,b)$ are long paths, they are in $R,$ and we obtain $p_i\in RJ\cap JR.$  Therefore $a(RJ\cap JR)_4b=\langle  p_0, p_1 \rangle.$ 
 
 Since $(R^2)_4=R_2R_2,$ we obtain we obtain $a(R^2)_4b =\sum_{i\in \{0,1\}} aR_2c_iR_2b.$ However, $c_iR_2b=0,$ because there is only one path of length $2$ from $c_i$ to $b$. Therefore $a(R^2)_4b=0.$
 
Using that  $(JRJ)_4 = J_1R_2J_1$ and that all paths of length $4$ from $1$ to $3$ have the form $(a,b,?,a,b),$ we obtain $aJ_1R_2J_1b=(a,b)R_2(a,b).$ There are only two paths from $b$ to $a$ of length $2,$ and they are shortest: $(b,c_i,a).$ Therefore $bR_2a =\langle  (b,c_0,a)-(b,c_1,a) \rangle.$ Hence $a(JRJ)_4b= \langle p_0-p_1\rangle.$ 
\end{example}
\begin{example} \label{example:computation2}
In this example we will similarly prove that $\MH_{3,4}(G,\ZZ)$ is nontrivial for the following undirected graph.
\begin{equation}
\begin{tikzpicture}[baseline]
\tikzset{
vertex/.style={circle,
inner sep=0pt, outer sep=1pt,
minimum width=3pt}}
\node[vertex] (a0) at (-1,1) {$c$};
\node[vertex] (a1) at (1,1) {$b_0$};
\node[vertex] (a2) at (1,-1) {$a$};
\node[vertex] (a3) at (-1,-1) {$b_1$};
\node[vertex] (b1) at (0.4,0.4) {$d_0$};
\node[vertex] (b2) at (0.4,-0.4) {$e$};
\node[vertex] (b3) at (-0.4,-0.4) {$d_1$};
\draw[black] 
(a0) -- (a1) -- (a2) -- (a3) -- (a0)
(a0) -- (b1) -- (b2) -- (b3) -- (a0)
(a1) -- (b1)
(a2) -- (b2)
(a3) -- (b3)
;
\end{tikzpicture}
\end{equation}
There are $4$ paths from $a$ to $e$ of length $4:$ $p_{ij}=(a,b_i,c,d_j,e),$ where $i,j\in \{0,1\}.$ We will prove that 
\begin{equation}
 \MH_{3,4}(G,\ZZ)_{a,e}\cong 
 \left( \frac{a(RJ\cap JR)e}{a(R^2+JRJ)e}\right)_4  = \frac{ \langle p_{00}, p_{01}, p_{10}, p_{11} \rangle}{ \langle 
p_{01} + p_{10}, p_{00},p_{11} \rangle }  \cong \ZZ .
\end{equation}
Since $(b_i,c,d_i)$ is a long path, we obtain $p_{i,i}\in JRJ.$ Since $(a,b_i,c,d_{1-i})$ and $(b_i,c,d_{1-i},e),$ we obtain $p_{i,1-i}\in JR\cap RJ.$ Therefore, $a(RJ\cap JR)e = \langle p_{00},p_{01},p_{10},p_{11} \rangle.$

Since $(R^2)_4=R_2R_2,$ and all the paths have the form $(a,?,c,?,e),$ we have $a(R^2)_4e=aR_2cR_2e.$ It is easy to see that $aR_2c =\langle (a,b_0,c) - (a,b_1,c) \rangle $
and $c R_2 e = \langle (c,d_0,e) - (c,d_1,e) \rangle.$ Since $((a,b_0,c) - (a,b_1,c)) \cdot ((c,d_0,e) - (c,d_1,e)) = p_{00} - p_{01} - p_{10} + p_{11},$ we obtain $a(R^2)_4e = \langle  p_{00} - p_{01} - p_{10} + p_{11} \rangle.$ 

Using the formulas $a(JRJ)_4e = \sum_{i,j} (a,b_i)R_2 (d_j,e),$ and  $b_i R_2 d_i = \langle (b_i,c,d_i) \rangle$ and $b_i R_2 d_{1-i} =0,$ we obtain $a(JRJ)_4e =\langle p_{00},p_{11} \rangle.$
\end{example}

\section{The second magnitude homology}

In this section, we use the Gruenberg formula to show that, for any finite digraph $G$, the second magnitude homology $\MH_{2,\ell}(G,\KK)$ is a free $\KK$-module with an explicit basis (Theorem~\ref{theorem:second}). Furthermore, we use this basis to obtain, for a given $\ell$, a description of digraphs $G$ such that $\MH_{2,k}(G) = 0$ for $k > \ell$ (Theorem~\ref{th:vanishing}). Using this description, we show that if $\MH_{2,k}(G) = 0$ for $k > \ell$, then the $\ell$-fundamental category of $G$ is thin (Proposition~\ref{prop:tau-thin}). As a corollary, we obtain that for a diagonal undirected graph $G$, the associated 2-dimensional CW-complex ${\sf CW}^2(X)$ is simply connected (Corollary~\ref{cor:cw}). Finally, we show that our methods provide an alternative proof of the Asao-Hiraoka-Kanazawa theorem concerning the girth of a diagonal digraph (Proposition~\ref{prop:girth}).

\subsection{Basis of the second magnitude homology} Note that for a digraph $G$ we have $R\subseteq J^2.$ Therefore, Theorem \ref{theorem:Gruenberg_for_magnitude} implies that there is an isomorphism over an arbitrary commutative ring $\KK$ 
\begin{equation}\label{eq:iso:Grumberg_2}
\MH_{2,\ell}(G)\cong \left(\frac{R}{JR + RJ}\right)_\ell. 
\end{equation}
In this subsection we use this isomorphism to show that $\MH_{2,\ell}(G)$ is a free  $\KK$-module, and describe its basis. 

If there is a path $p=(p_0,\dots,p_n)$ we say that $p$ connects $p_0$ and $p_n.$ We say that two paths $(p_0,\dots,p_n)$ and $(q_0,\dots,q_m)$ connect the same vertices if $p_0=q_0$ and $p_n=q_m.$

\begin{definition}[Congruences and ideals for paths] Consider the category of paths in a digraph $G,$ whose objects are vertices, morphisms are paths and composition is defined by the concatenation. We say that some equivalence relation on the set of paths is a congruence if it is a congruence on the category of paths i.e. if $p\sim \tilde p,$ then the $p$ and $\tilde p$ connect the same vertices and $qpq'\sim q\tilde pq'$ whenever concatenation is defined. We say that some subset of paths $I$ is an ideal if it is an ideal in the category of paths i.e. $p\in I$ implies $qpq'\in I$ whenever the concatenation is defined. For example, the set of long paths is an ideal.
\end{definition}

\begin{definition}($\ell$-short congruence) The 
$\ell$-short congruence is the minimal congruence on the set of paths such that any two shortest paths connecting the same two vertices at distance at most $\ell$ are equivalent. We say that two paths differ by an $\ell$-short move if they can be presented as compositions $psp'$ and $p\tilde sp'$ where $s$ and $\tilde s$ are two different shortest paths of length at most $\ell$ connecting the same vertices. 
Then two paths are $\ell$-short congruent if and only if one of them can be obtained from another by a sequence of $\ell$-short moves. Note that two $\ell$-shortly congruent paths have the same length. 
\end{definition}

We denote by $\tP_\ell$ the set of all paths of length $\ell$, and by $\overline\tP_\ell$ the quotient of $\tP_\ell$ by the $(\ell-1)$-short congruence 
\begin{equation}
\overline\tP_\ell = \tP_\ell/\sim_{(\ell-1)\text{-short}}.
\end{equation}
Consider two ideals $R^{\tS},R^{\tL} \subseteq R,$ where $R^{\tS}$ is generated by the differences of shortest paths $s-\tilde s$ connecting the same vertices, and $R^{\tL}$ is generated by long paths. Then 
\begin{equation}
R=R^{\tS}+R^{\tL}.
\end{equation}

\begin{lemma} There is an isomorphism 
\begin{equation}\label{eq:iso:sh}
\left( \frac{\KK G}{JR^{\tS} + R^{\tS}J} \right)_\ell \cong \KK\cdot \overline{\tt P}_\ell
\end{equation}
induced by the canonical projection ${\tt P}_\ell\epi \overline{\tt P}_\ell.$
\end{lemma}
\begin{proof}
It is easy to see that the ideal $JR^{\tS} + R^{\tS}J$ is spanned over $\mathbb K$ by the differences $psp'-p\tilde s p',$ where  $s,\tilde s$ are two different shortest paths between two vertices $x,x'$ and $p,p'$ are any paths so that at least one of them has length at least $1.$  Therefore $(JR^{\tS} + R^{\tS}J)_\ell$ is spanned by such differences $psp'-p\tilde s p'$ having ${\rm len}(p)+{\rm len}(s)+{\rm len}(p')=\ell$ and ${\rm len}(p)+{\rm len}(p')\geq 1,$ and hence,  $d(x,x')={\rm len}(s)\leq l-1.$ This proves the claim. 
\end{proof}

Denote by $\tS_\ell$ the set of shortest paths of length $\ell$ and by  $\tL_\ell$ the set of long paths of length $\ell$
\begin{equation}
\tP_\ell = \tS_\ell \sqcup \tL_\ell.
\end{equation}
Note that a shortest path cannot be $(\ell-1)$-shortly congruent to a long path. Therefore the images of these sets in $\overline\tP_\ell$ do not intersect, and if we denote them by $\overline\tS_\ell$ and $ \overline\tL_\ell,$ we obtain
\begin{equation}
\overline\tP_\ell = \overline\tS_\ell \sqcup \overline\tL_\ell.
\end{equation}
We say that a long path is minimal if all its proper subpaths are shortest. Otherwise it is called non-minimal long path. The set of non-minimal long paths of length $\ell$ is denoted by  $\tN_\ell$ 
\begin{equation}
\tN_\ell \subseteq  \tL_\ell.
\end{equation}
Its image in $\overline\tL_\ell$ is denoted by $\overline\tN_\ell$
\begin{equation}
\overline\tN_\ell \subseteq  \overline\tL_\ell.
\end{equation}
Therefore  $\overline\tN_\ell$ consists of $(\ell-1)$-short congruence classes of long paths of length $\ell$ containing a non-minimal long path. Note that a minimal long path of length $\ell$ can be $(\ell-1)$-shortly congruent to a non-minimal long path. For instance, if we take $\ell=3$ and the following graph, 
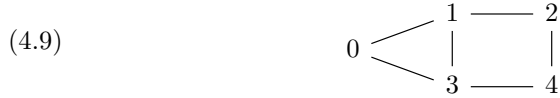
\begin{equation}
\begin{tikzcd}[row sep=0mm]
& 1 \ar[r,-] \ar[dd,-]  & 2 \ar[dd,-] \\
0 \ar[ru,-] \ar[rd,-] & && \\
& 3 \ar[r,-] & 4 
\end{tikzcd}
\end{equation}
then the minimal long path  $(0,1,2,4)$ is $2$-shortly congruent to the non-minimal long path $(0,1,3,4).$ Therefore some elements of $\overline\tN_\ell$ can be presented both by a minimal long path and by a non-minimal long path.

\begin{lemma}\label{lemma:iso_JR+RJ}
The isomorphism \eqref{eq:iso:sh} induces an isomorphism 
\begin{equation}\label{eq:lemma:iso:KG/JR+RJ}
\left( \frac{\KK G}{JR + RJ} \right)_\ell \cong \KK\cdot (\overline\tP_\ell\setminus \overline\tN_\ell).
\end{equation}
\end{lemma}
\begin{proof}
The ring $\KK G/(JR + RJ)$ is isomorphic to the quotient of  $\KK G/(JR^{\tS} + R^{\tS}J)$ by the image of $JR^{\tL} + R^{\tL}J.$ The ideal $JR^{\tL} + R^{\tL}J$ is spanned over $\KK$ by non-minimal long paths. This implies the assertion. 
\end{proof}

\begin{theorem}\label{theorem:second} Let $G$ be a finite digraph, $\KK$ be a commutative ring and $\ell\geq 0$ be an integer. Assume that for any vertices $x,y$ at a distance $d(x,y)=\ell$ we have a chosen shortest path $s_{x,y}$ from $x$ to $y.$
Then the $\KK$-module 
\begin{equation}
\MH_{2,\ell}(G)\cong \left(\frac{R}{JR+RJ}\right)_\ell   
\end{equation}
is isomorphic to a free submodule of $\KK\cdot (\overline\tP_\ell\setminus \overline\tN_\ell)$ freely generated by elements of the form:
\begin{enumerate}
    \item $[s]-[s_{x,y}],$ where $s$ is a shortest path from $x$ to $y$ such that $[s]\neq [s_{x,y}];$  
    \item $[p],$ where $[p]\in  \overline\tL_\ell\setminus \overline\tN_\ell;$
\end{enumerate}
and the isomorphism is induced by \eqref{eq:lemma:iso:KG/JR+RJ}.
\end{theorem}
\begin{proof}
Lemma \ref{lemma:iso_JR+RJ} and the isomorphism \eqref{eq:iso:Grumberg_2} imply that $\MH_{2,\ell}(G)$ is isomorphic to the image of $R_\ell$ in the module $\KK\cdot (\overline\tP_\ell\setminus \overline\tN_\ell).$ Since  $R=R^{\tS}+R^{\tL},$ this image is equal to the sum of the image of $R^{\tS}_\ell$ and  the image of $R^{\tL}_\ell.$

Let us first describe the image of $R^{\tS}_\ell.$ The module $R^{\tS}_\ell$ is spanned by the differences of the form $psp'-p\tilde sp',$ where $s$ and $\tilde s$ are two different shortest paths connecting the same vertices and ${\rm len}(p)+{\rm len}(s)+{\rm len}(p')=\ell$. If ${\rm len}(p)+{\rm len}(p')\geq 1,$ then the image of $psp'-p\tilde sp'$ in $\KK\cdot (\overline\tP_\ell\setminus \overline\tN_\ell)$ is trivial, because $psp'$ is $(\ell-1)$-shortly congruent to $p\tilde s p'.$ Then the image of $R^{\tS}_\ell$ is spanned by the differences $[s]-[\tilde s],$ where $s$ and $\tilde s$ are two different shortest paths between some vertices at a distance $\ell.$ Hence, it is easy to see that elements of the form $[s]-[s_{x,y}]$ form a basis of the image of $R^{\tS}$. In particular, we obtain that the image of $R^{\tS}$ lies in $\KK\cdot Q^{\sf sh}_\ell.$

The module $R^{\tL}_\ell$ is spanned by  long paths. Therefore, its image in  $\KK\cdot (\overline\tP_\ell\setminus \overline\tN_\ell)$ is the free module $\KK\cdot (\overline\tL_\ell\setminus \overline\tN_\ell).$ Since $\overline\tP_\ell\setminus \overline\tN_\ell =\overline\tS_\ell \sqcup (\overline\tL_\ell\setminus \overline\tN_\ell),$ and the image of $R^{\tS}_\ell$ lies in $\overline\tS_\ell,$ the images of $R^{\tS}_\ell$ and $R^{\tL}_\ell$ intersect trivially. Therefore, the image of $R$ is the direct sum of images of $R^{\tS}$ and $R^{\tL},$ and its basis is the union of bases of the images.  
\end{proof}

For two vertices $x,y$ we denote by $\overline\tP_\ell(x,y),\overline\tS_\ell(x,y), \overline\tL_\ell(x,y), \overline\tN_\ell(x,y)$ the subsets of $\overline\tP_\ell,\overline\tS_\ell, \overline\tL_\ell, \overline\tN_\ell$ consisting of the equivalence classes of paths connecting $x$ and $y.$

\begin{corollary}\label{cor:ranks}
Let $x,y$ be two vertices of a digraph $G$ and $\ell\geq 0$ be an integer. Then $\MH_{2,\ell}(G)_{x,y}$ is a free $\KK$-module, whose rank is described by the formula 
\begin{equation}
{\rm rk}_\KK (\MH_{2,\ell}(G)_{x,y}) =\begin{cases}
|\overline\tL_\ell(x,y)\setminus \overline\tN_\ell(x,y)|, & d(x,y)<\ell,\\
|\overline\tS_\ell(x,y)|-1, & d(x,y)=\ell,\\
0, & d(x,y) > \ell.
\end{cases}
\end{equation}
\end{corollary}

\begin{corollary}
If $\KK=\ZZ,$ the abelian group $\MH_{2,\ell}(G,\ZZ)$ is torsion free for any finite digraph $G.$
\end{corollary}

The following corollary can be easily proved directly from the definition of the magnitude homology, but it also easily follows from our description. 

\begin{corollary}
The rank of the free $\KK$-module $\MH_{2,2}(G)$ is equal to the number of paths of length $2$ minus the number of pairs of vertices at distance $2.$
\end{corollary}

\subsection{Vanishing of the second magnitude homology}

\begin{proposition}\label{prop:vanishing_fixed_l}
Let $G$ be a finite digraph and $\ell\geq 2$ be an integer. Then  $\MH_{2,\ell}(G)=0$ if and only if the following conditions hold
\begin{enumerate}
    \item any two shortest paths of length $\ell$ connecting the same vertices are $(\ell-1)$-shortly congruent;
    \item any minimal long path of length $\ell$ is $(\ell-1)$-shortly congruent to a non-minimal long path.
\end{enumerate}
\end{proposition}
\begin{proof}
It follows from Corollary \ref{cor:ranks}.  
\end{proof}

A long path is called \emph{$\ell$-reducible} if it contains a long subpath of length at most $\ell.$ Then a non-minimal long path of length $\ell$ is an $(\ell-1)$-reducible path of length $\ell.$ All long paths of length at most $\ell$ are $\ell$-reducible. A long path is called \emph{$\ell$-quasi-reducible} if it is $\ell$-shortly congruent to an $\ell$-reducible path. Note that the set of all $\ell$-reducible paths is an ideal, and the set of all $\ell$-quasi-reducible paths is an ideal. 

We say that a digraph $G$ satisfies property $(\mathcal{V}_\ell)$ if 
 
 \medskip

 \begin{itemize}
     \item[$(\mathcal{V}_\ell)$] \it any two shortest paths connecting the same vertices are $\ell$-shortly congruent; and any long path is $\ell$-quasi-reducible.
 \end{itemize}

 \medskip

\begin{theorem}\label{th:vanishing} Let $G$ be a finite digraph and $\ell\geq 1$ be an integer. Then $\MH_{2,k}(G)=0$ for any $k>\ell$ if and only if $G$ satisfies property $(\mathcal{V}_\ell).$ 
\end{theorem}
\begin{proof} If two paths are $\ell$-shortly congruent, then they are $(k-1)$-shortly congruent for any $k>\ell.$ Therefore $(\mathcal{V}_\ell)$ implies that $\MH_{2,k}(G)=0$ for $k>\ell$ by Proposition \ref{prop:vanishing_fixed_l}.

Now assume that $\MH_{2,k}(G)=0$ for $k>\ell$ and prove property $(\mathcal{V}_\ell).$ First we prove that two shortest paths connecting the same vertices are $\ell$-shortly congruent. Note that two shortest paths of length $\leq \ell$ are $\ell$-shortly congruent by the definition. So we can assume that the shortest paths have the length $\ell+n$ for some $n\geq 1.$  
Proposition \ref{prop:vanishing_fixed_l} implies that for any $n\geq 1$ any two shortest paths of length $\ell+n$ connecting the same vertices are $(\ell+n-1)$-shortly congruent. 
We need to prove that they are $\ell$-shortly congruent. The proof is by induction on $n$. The base of induction for $n=1$ is obvious. Let us prove the induction step. Assume that $s$ and $\tilde s$ are two shortest paths of length $\ell+n.$ Since they are $(\ell+n-1)$-shortly congruent, there is a sequence of shortest paths $s=s_0,\dots,s_m=\tilde s$ such that for any $0\leq i\leq m-1$ we have $s_i=p_i t_i p'_i$ and $s_{i+1}=p_i\tilde t_i p_i'$ such that $t_i$ and $\tilde t_i$ are shortest paths of length  $\leq \ell+k-1.$ By the inductive assumption,  $t_i$ and $\tilde t_i$ are $\ell$-shortly congruent. Therefore, $s_i$ and $s_{i+1}$ are $\ell$-shortly congruent. Hence $s$ and $\tilde s$ are $\ell$-shortly congruent. 

Let us prove that any long path is $\ell$-quasi-reducible. Any long path of length at most $\ell$ is $\ell$-reducible. So we can assume that the long path has length $\ell+n$ for some $n\geq 1.$ Since we already know that any two shortest paths connecting the same vertices are $\ell$-shortly congruent, the $(\ell+n-1)$-short congruence coincides with the $\ell$-short congruence for any $n\geq 1.$ Then Proposition \ref{prop:vanishing_fixed_l} implies that for any $n\geq 1$ any  long path $p$ of length  $\ell+n$ is $\ell$-shortly congruent to an $(\ell+n-1)$-reducible long path $\tilde p$. Choose a long subpath $\hat p$ of length $(\ell+n-1)$ in $\tilde p.$ By inductive assumption $\hat p$ is $\ell$-quasi-reducible. Since the set of $\ell$-quasi-reducible paths is an ideal, we obtain that $\tilde p$ and $p$ are also $\ell$-quasi-reducible. 
\end{proof}

\begin{corollary}
Any diagonal digraph satisfies $(\mathcal{V}_2).$
\end{corollary}

\subsection{Thinness of the fundamental category} 

In this subsection, for a digraph $G$ and a positive integer $\ell,$ we define a category $\tau^\ell(G)$ that we call  $\ell$-fundamental category of $G$.
It is closely related with the $\ell$-fundamental groupoid $\Pi^\ell(G)$ and $\ell$-fundamental group $\pi^\ell_1(G)$ studied in \cite{di2024path}. We will show that if there is an integer $\ell$ such that $\MH_{2,k}(G)=0$ for any $k>\ell,$ then $\tau^\ell(G)$ is a thin category.  

We denote by $\tau$ the functor of the fundamental category of a simplicial set $\tau:{\sf sSets} \to {\sf Cat},$ which is the left adjoint to the nerve functor ${\sf Nrv}:{\sf Cat} \to {\sf sSets}$. It can be explicitly described as follows (see \cite[\S II.4.1]{gabriel2012calculus},  \cite[\S 1.3]{joyal2008notes}). 
Let $X$ be a simplicial set with face maps denoted by $\partial_i$ and degeneracy maps denoted by $\sigma_i.$ Then $\tau(X)$ is the quotient of the free path category of its $1$-skeleton (treated as a quiver) by relations $\partial_0t \circ \partial_2t=\partial_1t,$ where $t$ is a $2$-simplex, and relations $\sigma_0 x=1_x,$ where $\alpha$ is a $1$-simplex. The composition is defined by the concatenation, and the identity morphisms are degenerate $1$-simplices $1_x=\sigma_0x$.  The groupoid associated with $\tau(X)$ is the fundamental groupoid of $X$ \cite[\S II.7.1]{gabriel2012calculus}
\begin{equation}
\Pi(X) = \tau(X)^{\sf gpd}.
\end{equation}

Following \cite{di2024path}, we denote by $\NN(G)$ the nerve of a digraph $G,$ which is a simplicial set, whose $n$-simplices are tuples of vertices $(x_0,\dots,x_n)$ such that for any $0\leq i\leq n-1$ there is a path from $x_i$ to $x_{i+1}.$ The face maps are defined by deleting a vertex, and the degeneracy maps are defined by doubling of vertices. This simplicial set has a natural filtration by simplicial subsets
\begin{equation}
    \NN^1(G) \subseteq \NN^2(G) \subseteq \dots \subseteq \NN(G),
\end{equation}
where $n$-simplices of $\NN^\ell(X)$ are tuples of vertices $(x_0,\dots,x_n)$ such that $|x_0,\dots,x_n|\leq \ell.$ The $\ell$-fundamental category of $G$ is defined by 
\begin{equation}
\tau^\ell(G) = \tau(\NN^\ell(G)).
\end{equation}
Then the $\ell$-fundamental groupoid $\Pi^\ell(G)=\Pi(\NN^\ell(G))$ studied in \cite{di2024path} is the groupoid associated with the $\ell$-fundamental category 
\begin{equation}
\Pi^\ell(G) = \tau^\ell(G)^{\sf gpd}
\end{equation}
and the $\ell$-fundamental group is defined by $\pi^\ell_1(G,x)=\Pi^\ell(G)(x,x).$

The $\ell$-fundamental category can be described in terms of equivalence classes of paths.
\begin{definition}[$\tau^\ell$-congruence] 
$\tau^\ell$-congruence is the minimal congruence on the set of paths such that any two paths connecting the same vertices  of (possibly different) lengths at most $\ell$ are equivalent. 
\end{definition}

Note that $\ell$-shortly congruent paths are $\tau^\ell$-congruent. However, $\tau^\ell$-congruent paths can have different lengths. 

\begin{example}
 In the following graph 
\begin{equation}\label{eq:graphs:counter}
\begin{tikzcd}[row sep=0mm]
&1 \ar[r,-] \ar[dd,-]  & 3 \ar[dd,-] \ar[rd,-]\\
0\ar[ru,-] \ar[rd,-] & && 5\\
&2\ar[r,-] & 4\ar[ru,-] 
\end{tikzcd}
\end{equation}
all the paths $(0,1,3,5),$ $(0,1,3,4,5),$ and $(0,2,4,5)$ are $\tau^2$-congruent, but all of them are not $2$-shortly congruent to each other.   
\end{example}

\begin{example}
Another example is given by the following graph. 
\begin{equation}
\begin{tikzcd}
0 \ar[r,-] \ar[d,-]\ar[dr,-]  & 1 \ar[d,-]\\
2\ar[r,-] & 3 
\end{tikzcd}
\end{equation}
The paths $(0,1,3)$ and $(0,2,3)$ are $\tau^2$-congruent but not $2$-shortly congruent, because they are not shortest.  
\end{example}

The following proposition is an analogue of the description of the fundamental groupoid $\Pi^\ell(G)$ given in \cite[Prop.2.2]{di2024path}.

\begin{proposition}\label{prop:tau^l-description}
The fundamental category $\tau^\ell(G)$ of a digraph $G$ is naturally isomorphic to the category, whose objects are vertices of $G$, whose morphisms from a vertex $x$ to a vertex $y$ are $\tau^\ell$-congruence classes of paths from $x$ to $y,$ and the composition is defined by concatenation. 
\end{proposition}
\begin{proof}
In the proof, denote by $\mathcal C$ the category, whose objects are vertices of $G$ and morphisms are $\tau^\ell$-congruence classes of paths. The $\tau^\ell$-congruence class of a path $(x_0,\dots,x_n)$ will be denoted by $[x_0,\dots,x_n].$ We need to prove that $\tau^\ell(G)\cong \mathcal C.$ The plan of the proof is the following: we construct two functors $\Psi:\tau^\ell(G) \to \mathcal C$ and $\Phi:\mathcal C \to \tau^\ell(G)$ and prove that $\Psi\Phi = {\sf Id}$ and $\Phi\Psi = {\sf Id}.$

Before we start constructing the functors, we need to note that, if $(x_0,\dots,x_n)$ is a path of length $n\leq \ell$ in $G,$ then  there is the following equation in $\tau^\ell(G)$
\begin{equation}\label{eq:path-simpl}
((x_0,x_1),(x_1,x_2),\dots,(x_{n-1},x_n)) =((x_0,x_n)). 
\end{equation}
Indeed, for any $0\leq i\leq n-2$ we have $d(x_i,x_{i+1})+d(x_{i+1},x_n)\leq \ell,$ and hence $(x_i,x_{i+1},x_n)$ is a $2$-simplex in $\NN^\ell(G).$ Therefore $((x_i,x_{i+1}),(x_{i+1},x_n))=((x_i,x_n)),$ and we obtain \eqref{eq:path-simpl} by induction.  

The functor $\Phi:\mathcal C\to \tau^\ell(G)$ is defined identically on objects and sends $[x_0,\dots,x_n]$ in $G$ to the class of the path $((x_0,x_1),\dots,(x_{n-1},x_n))$ of $1$-simplices of $\NN^\ell(X).$ In order to show that $\Phi$ is well defined, we need to check that $\tau^\ell$-congruent paths have the same images. 
It is sufficient to show that two different paths $(x_0,\dots,x_n)$ and $(y_0,\dots,y_m)$ of lengths $n,m\leq \ell$ connecting the same vertices $x_0=y_0$ and $x_n=y_m$ have the same images in $\tau^\ell(G)$. It follows from \eqref{eq:path-simpl}. Hence the functor $\Phi$ is well defined. 

The functor $\Psi:\tau^\ell(G)\to \mathcal C$ is defined identically on objects. If $(x,y)$ is a $1$-simplex of $\NN^\ell(G),$ then $d(x,y)\leq n$ and we define $\Psi((x,y))=[x_0,\dots,x_n],$ where $(x_0,\dots,x_n)$ is a path from $x$ to $y$ of any length $n$ such that $d(x,y)\leq n\leq \ell.$  Since $n\leq \ell,$ any two such paths are $\tau^\ell$-congruent, and the definition does not depend on the choice. If we have a $2$-simplex $(x,y,z)$ of $\NN^\ell(G),$ then $d(x,y)+d(y,z)\leq \ell$ and the concatenation of the shortest path from $x$ to $y$ and the shortest path from $y$ to $z$ has length at most $\ell.$ Therefore $\Psi((y,z)) \circ \Psi((x,y)) = \Psi((x,z)).$ Therefore, the definition on objects and $1$-simplices gives a well defined functor $\Psi:\tau^\ell(G)\to \mathcal C.$

The fact that $\Psi \Phi={\sf Id}$ is straightforward. Let us prove that $\Phi \Psi = {\sf Id}.$ For objects it is obvious. Let us prove it for morphisms. It is sufficient to check it for $1$-simplices. If we have a $1$-simplex $(x,y)$ in $\NN^\ell(G),$ then $\Phi \Psi((x,y))=((x_0,x_1),(x_1,x_2),\dots,(x_{n-1},x_n)),$ where $(x_0,\dots,x_n)$ is a path of length $n$ from $x$ to $y$ such that $d(x,y)\leq n\leq \ell.$ Then  \eqref{eq:path-simpl} implies $\Phi\Psi((x,y))=(x,y).$
\end{proof}

A category is called thin if all its hom-sets have cardinality at most one. Thin categories are in one-to-one correspondence with preorders. 

\begin{proposition}\label{prop:tau-thin}
If $G$ is a finite digraph and $\ell\geq 1$ is an integer such that $\MH_{2,k}(G)=0$ for any $k>\ell,$ then $\tau^\ell(G)$ is thin. 
\end{proposition}
\begin{proof} Proposition \ref{prop:tau^l-description} implies that we need to prove that any two paths connecting the same vertices are $\tau^\ell$-congruent. 
By Theorem \ref{th:vanishing} any two shortest paths are $\ell$-shortly congruent and any long path is $\ell$-shortly congruent to an $\ell$-reducible long path. 
Since $\ell$-short congruence implies $\tau^\ell$-congruence, any two shortest paths connecting the same vertices are $\tau^\ell$-congruent. Note that any $\ell$-reducible path is $\tau^\ell$-congruent to a path of shorter length. Then any long path is $\tau^\ell$-congruent to a path of shorter length. Therefore by induction we obtain that any long path is $\tau^\ell$-congruent to a shortest path. Hence all paths between two vertices are $\tau^\ell$-congruent.
\end{proof}

\begin{remark} The condition that $\tau^\ell(G)$ is thin is not equivalent to the fact that $\MH_{2,k}(G)=0$ for any $k>\ell.$ For example, 
for the graph \eqref{eq:graphs:counter} the category $\tau^2(G)$ is thin, but $\MH_{2,3}(G)\neq 0,$ because paths $(0,1,3,5)$ and $(0,1,3,5)$ are two shortest paths of length $3$ which are not $2$-shortly congruent.
\end{remark}

\subsection{Undirected graphs}

By an undirected graph we mean a directed graph such that for any arrow $(x,y)$ the pair $(y,x)$ is also an arrow. 

\begin{lemma}\label{lemma:long}
Let $G$ be an undirected graph and $x,y$ be its vertices. Then 
\begin{enumerate}
    \item a minimal long path from $x$ to $y$ has length either $d(x,y)+1$ or $d(x,y)+2;$
    \item any minimal long path of length $d(x,y)+2$ is $(d(x,y)+1)$-shortly congruent to a $2$-reducible path.
\end{enumerate}

\end{lemma}
\begin{proof} Set $n=d(x,y).$
Assume that $p=(p_0,\dots,p_m)$ is a path from $x$ to $y$ of length $m> n+2.$ The triangle inequality implies that $d(p_0,p_{m-1})\leq d(p_0,p_m)+d(p_m,p_{m-1})=n+1<m-1.$ Then $(p_0,\dots,p_{m-1})$ is a long path. Hence $p$ is not a minimal long path.

Assume that $p=(p_0,\dots,p_{n+2})$ is a minimal long path from $x$ to $y.$ Choose some shortest path $s=(s_0,\dots,s_n)$ from $x$ to $y.$  Then $(p_0,\dots,p_{n+1})$ is a shortest path from $x$ to $p_{n+1}$ and $d(x,p_{n+1})=n+1.$ Hence $(s_0,\dots,s_n,p_{n+1})$ is also a shortest path from $x$ to $p_{n+1}.$ Therefore, 
$(p_0,\dots, p_n, p_{n+1},p_{n+2})$ is $(n+1)$-shortly congruent to $(s_0,\dots,s_n,p_{n+1},p_{n+2}).$ Since $y=s_n=p_{n+2}$ and $n>0,$ we obtain that $(s_0,\dots,s_n,p_{n+1},p_{n+2})=(s_0,\dots,y,p_{n+1},y)$ is $2$-reducible. 
\end{proof}

\begin{proposition} Let $G$ be a finite undirected graph and $x,y$ be its vertices. Then 
\begin{enumerate}
    \item if $x\neq y$ and $\MH_{2,\ell}(G)_{x,y}\neq 0,$ then $\ell\in \{d(x,y),d(x,y)+1\};$
    \item if $\MH_{2,\ell}(G)_{x,x}\neq 0,$ then $\ell=2.$
\end{enumerate}
\end{proposition}
\begin{proof} It follows from Corollary \ref{cor:ranks} and Lemma  \ref{lemma:long}. 
\end{proof}

\begin{proposition}
If $G$ is a finite undirected graph and $\ell\geq 2$ is an integer such that $\MH_{2,k}(G)=0$ for $k>\ell,$ then $\pi_1^\ell(G,x)=1$ for any vertex $x$. 
\end{proposition}
\begin{proof}
For an undirected digraph $G$ and $\ell\geq 2$ the $\ell$-fundamental category   $\tau^\ell(G)$ is a groupoid. Indeed for any arrow $(x,y)$ in $G$ the path $(x,y,x)$ is $\tau^\ell$-congruent to the trivial path $(x),$ and similarly $(y,x,y)$ is $\tau^\ell$-congruent to $(y).$ Hence $(x,y)$ is invertible. Therefore $\tau^\ell(G) = \Pi^\ell(G)$ and the assertion follows from Proposition \ref{prop:tau-thin}. 
\end{proof}

Grigor’yan--Lin--Muranov--Yau \cite{grigor2014homotopy} define the fundamental group of a digraph $G$ with base vertex $x,$ that we denote by $\pi_1^{\sf GLMY}(G,x).$ The abelianization of this group is the first path homology. Later this definition was extended by Grigor’yan--Jimenez--Muranov to the notion of  fundamental groupoid \cite{grigor2018fundamental} that we denote by $\Pi^{\sf GJM}(G)$.  It is known \cite[Th.2.4]{di2024path} that 
\begin{equation}
\pi_1^{\sf GLMY}(G,x)\cong \pi_1^2(G,x).
\end{equation}

\begin{corollary}\label{cor:fundamental_diagonal}
If $G$ is a finite undirected graph such that $\MH_{2,k}(G,\ZZ)=0$ for $k>2$, then the group $\pi_1^{\sf GLMY}(G,x)$ is trivial for any vertex $x.$
\end{corollary}

For an undirected digraph $G$ we denote by ${\sf CW}^2(G)$ a two-dimensional CW-complex obtained from the geometric realization of $G$ by attaching 2-cells to all its squares and triangles. Then it is known that  
\begin{equation}
\pi_1^{\sf GLMY}(G,x) = \pi_1({\sf CW}^2(G),x)
\end{equation}
(see \cite[Corollary 4.5]{grigor2018fundamental}). 

\begin{corollary}\label{cor:cw}
If $G$ is a finite  undirected connected graph such that $\MH_{2,k}(G,\ZZ)=0$ for $k>2$, then ${\sf CW}^2(G)$ is simply connected. 
\end{corollary}

\subsection{Application: girth}
In order to illustrate our methods, we give an alternative proof of a theorem proved by Asao--Hiraoka--Kanazawa  \cite{asao2024girth}. 
By a cycle in an undirected graph we mean a path $(x_0,\dots,x_n)$ such that $n\geq 3,$ $x_0=x_n$ and all the vertices $x_0,\dots,x_{n-1}$ are distinct. If $e$ is an edge of $G$ we denote by ${\sf gir}_e(G)$ the length of the shortest cycle containing $e.$ If such a cycle does not exist, ${\sf gir}_e(G)=\infty.$ By the definition we have ${\sf gir}_e(G)\geq 3$ for any $G$ and $e.$

\begin{proposition}[{cf. \cite[Th.1.5]{asao2024girth}}] \label{prop:girth} Let $G$ be a finite undirected graph and $e$ be its edge. If ${\sf gir}_e(G)<\infty,$ then $\MH_{2,\ell}(G)\neq 0,$ where $\ell = \lfloor ({\sf gir}_e(G)+1)/2 \rfloor.$ 
\end{proposition}
\begin{proof} Set $g={\sf gir}_e(G).$
In this proof, for a path $q$ we denote by $q_i$ its $i$-th vertex $q=(q_0,\dots,q_n)$ and denote by $q^{-1}$ the opposite path $(q_n,\dots,q_0)$. We will also introduce some terminology that we will use only in this proof. We say that a path $q$ of length $n$ is a quasi-cycle if $q_n=q_0$ and $q_1\neq q_{n-1}.$ It is easy to see that from any quasi-cycle $q$ by deleting some of its vertices we can obtain a cycle $c$ of length $m \leq n$ such that $c_0=q_0$ and $c_1=q_1,c_{m-1}=q_{n-1}.$ We say that a quasi-cycle $q$ begins from the edge $(q_0,q_1).$ Therefore, the condition $g={\sf gir}_e(G)$ implies that any quasi-cycle $q$ that begins from $e$ has length at least $g.$

Denote by $c$ some shortest cycle that begins from $e.$ Consider the path $p:=(c_0,\dots,c_\ell),$ where $\ell=\lfloor (g+1)/2 \rfloor.$

We claim that for any path $p'$  which is $(\ell-1)$-shortly congruent to $p$ we have $p'_1=c_1.$ 
Let us prove this. Assume the contrary: let $p'$ be a path which is $(\ell-1)$-shortly congruent to $p$ and $p'_1\neq c_1.$ Then there exists a path $\tilde p$ which differs from $p'$ by an $(\ell-1)$-short move such that $\tilde p_1=c_1$. Since $p'_1\neq \tilde p_1$ and the paths differ by an $(\ell-1)$-short move, we obtain $p'_{\ell-1}=\tilde p_{\ell-1}.$ Therefore $(\tilde p_0,\dots,\tilde p_{\ell-1}, p'_{\ell-1},\dots,p'_0)$ is a quasi-cycle of length $2(\ell-1)<g$ that begins from $e,$ which is a contradiction. 

Assume that $g=2\ell.$  Then $\ell\geq 2$ and $d(c_0,c_\ell)=\ell$ because otherwise there would be a shorter quasi-cycle beginning by $e.$ Therefore we obtain that there are two shortest paths $p$ and $(c_{g},c_{g-1},\dots, c_\ell)$ which are not $(\ell-1)$-shortly congruent. Then Proposition \ref{prop:vanishing_fixed_l} implies that $\MH_{2,\ell}(G)\neq 0.$ 

Assume that $g=2\ell-1.$ Then $\ell\geq 2$ and  $d(c_0,c_\ell)=d(c_0,c_{\ell-1})=d(c_1,c_\ell)=\ell-1$ because otherwise there would be a shorter cycle containing $e.$ 
Then $p$ is a minimal long path of length $\ell.$ We claim that $p$ is not $(\ell-1)$-shortly congruent to a non-minimal long path. Assume the contrary, that there is a non-minimal long path $p'$ which is $(\ell-1)$-shortly congruent to $p.$  The fact that it is non-minimal implies that either $d(p'_0,p'_{\ell-1})\leq \ell-2$ or $d(p'_1,p'_{\ell})\leq \ell-2.$ As we proved above,  $p'_1=c_1,$ and hence $d(p'_1,p'_{\ell})=d(c_1,c_\ell)=\ell-1.$ This implies that  $d(c_0,p'_{\ell-1})\leq \ell-2$ and $d(c_1,p'_{\ell-1})=\ell-2.$ Take a shortest path $s$ from $c_0$ to $p'_{\ell-1}.$ Since $d(c_1,p'_{\ell-1})=\ell-2,$ we obtain $s_1\neq c_1.$ Therefore $(p'_0,\dots,p'_{\ell-1})s^{-1}$ is a quasi-cycle beginning from $e$ of length $ \ell-1 + d(c_0,p'_{\ell-1})\leq 2\ell-3<g.$ This gives a contradiction. Therefore, $p$ is a minimal long path which is not $(\ell-1)$-shortly congruent to a non-minimal long path. Hence Proposition \ref{prop:vanishing_fixed_l} implies that $\MH_{2,\ell}(G)\neq 0.$
\end{proof}

\section{Path cochain algebra}
In this section we consider the dg-algebra  $\Omega^\bullet(G)$ studied in GLMY-theory \cite[\S 3.4]{grigor2012homologies}, whose elements are called ``$d$-invariant forms''. We call this dg-algebra path cochain algebra. In  \cite{grigor2012homologies} this dg-algebra was considered over a field, but we will consider it over a commutative ring $\KK$, so as to make it more consistent with magnitude homology, which is typically considered over $\ZZ.$
The goal of this section is to prove that $\Omega^\bullet(G)$ is isomorphic to a quotient of the path algebra $\KK G$ by some quadratic relations (similar to the description done of the diagonal part of the magnitude cohomology \cite[Th.6.2]{hepworth2022magnitude}). 

Let us recall the definition of $\Omega^\bullet(G)$ following \cite{grigor2012homologies}. For a finite set $X,$ we consider the free module generated by $(n+1)$-tuples $\KK\cdot X^{n+1}$ and its dual $\Lambda^n(X)=\Hom_\KK(\KK\cdot X^{n+1},\KK).$ Then the basis of $\Lambda^n(X)$ dual to the basis of tuples in $\KK\cdot X^{n+1}$ is denoted by 
\begin{equation}
e^{x_0,\dots,x_n} \in \Lambda^n(X).
\end{equation}
These modules form a cochain complex $\Lambda^\bullet(X)$ with the differential $\partial:\Lambda^n(X)\to \Lambda^{n+1}(X)$ defined by 
\begin{equation}
\partial(e^{x_0,\dots,x_n}) = \sum_{i=0}^{n+1} \sum_{v\in X} (-1)^i e^{x_0,\dots,x_{i-1},v,x_{i},\dots,x_n}
\end{equation}
(see \cite[(2.15)]{grigor2012homologies}). We consider a chain subcomplex $\RR^\bullet(X)\subseteq \Lambda^\bullet(X),$ where $\RR^n(X)$ is generated by the elements of the dual basis $e^{x_0,\dots,x_n}$ such that $x_i\neq x_{i+1}$ for any $i.$ Then there is a structure of a dg-algebra on $\RR^\bullet(X)$ with the product defined by 
\begin{equation}
e^{x_0,\dots,x_n} e^{y_0,\dots,y_m} = 
\begin{cases}
e^{x_0,\dots,x_{n-1},y_0,\dots,y_m}, & y_0=x_n,\\
0, & \text{otherwise.}
\end{cases}
\end{equation}

Denote by $X^{\sf comp}$ the complete digraph on the finite set $X.$ Its arrows are all pairs $(x,y)$ such that $x\neq y.$ Its paths are all tuples  $(x_0,\dots,x_n)$ such that $x_i\neq x_{i+1}.$ We denote by $\KK X^{\sf comp}$ the path algebra of the complete digraph. 

\begin{lemma}\label{lemma:iso:R(X)}
There is an isomorphism of graded algebras 
\begin{equation}\label{eq:iso:R(X)}
\RR^\bullet(X) \cong   \KK X^{\sf comp}, \hspace{1cm} e^{x_0,\dots,x_n} \mapsto (x_0,\dots,x_n)
\end{equation}
which can be treated as an isomorphism of dg-algebras if we define the differential on $\KK X^{\sf comp}$ by the formula
\begin{equation}
\partial (x_0,\dots,x_n) = \sum_{i=0}^{n+1}\  \sum_{x_{i-1}\neq v\neq x_i} (-1)^i (x_0,\dots,x_{i-1},v,x_{i},\dots,x_n),
\end{equation}
where, for each $i,$ $v$ runs over all elements of $X$ such that $v\neq x_i,$ if $i\leq n,$ and $x_{i-1}\neq v$ for $1\leq i.$
\end{lemma}
\begin{proof}
The proof is straightforward. 
\end{proof}

Let $G=(V(G),A(G))$ be a digraph. Unlike in \cite{grigor2012homologies}, by a path in $G$ we mean a tuple of vertices 
$(x_0,\dots,x_n)$ such that $(x_i,x_{i+1})$ is an arrow. 
We set
\begin{equation}
\RR=\RR^\bullet(V(G))    
\end{equation}
and denote by $\cN^n$ a 
submodule $\RR^n$ generated by elements  $e^{x_0,\dots,x_n}$ such that $(x_0,\dots,x_n)$ is not 
a path of $G$. It is easy to see that $\cN$ is a homogeneous 
ideal of $\RR$ but it is not necessarily closed with respect to the differential.   
\begin{lemma}\label{lemma:iso:R/N}
The isomorphism \eqref{eq:iso:R(X)} induces an  isomorphism of graded algebras
\begin{equation}
    \RR/\cN \cong \KK G. 
\end{equation}
\end{lemma}
\begin{proof}
It follows immediately from the fact that the images of the elements  $e^{x_0,\dots,x_n},$ where $(x_0,\dots,x_n)$ is a path, form a basis of $\RR^n/\cN^n.$ 
\end{proof}

Consider the dg-ideal $\mathcal J$ of $\RR$ defined by
$\mathcal J^n = \cN^n + \partial(\cN^{n-1})$ and take the quotient dg-algebra
\begin{equation}
\Omega^\bullet(G) = \RR/\mathcal J, 
\end{equation}
which we call the algebra of path cochains. Then the path cohomology is defined by ${\rm PH}^{n}(G)=H^n(\Omega^\bullet(G)).$

\begin{lemma}\label{lemma:generators_of_J}
The ideal $\mathcal J/\cN$ of $\RR/\cN$ is generated by the elements $\tilde t_{x,y},$ which are indexed by pairs of vertices $(x,y)$ at a distance $d(x,y)=2,$ and defined by the formula 
\begin{equation}
\tilde t_{x,y}:=\sum_{v} e^{x,v,y} + \cN,     
\end{equation}
where the sum runs over all vertices $v$ such that $(x,v,y)$ is a path.
\end{lemma}
\begin{proof}
By the definition $\mathcal J/\cN$ is generated by the elements of the form $\partial(e^{x_0,\dots,x_n})+\cN,$ where $(x_0,\dots,x_n)$ is not a path. If $d(x,y)=2,$ then 
\begin{equation}
  \partial(e^{x,y}) = \sum_{v\in V(G)} (e^{v,x,y} - e^{x,v,y} + e^{x,y,v}).  
\end{equation}
Since $(v,x,y)$ and $(x,y,v)$ are not paths, we obtain  $e^{v,x,y},e^{x,y,v}\in \mathcal N$ and  $\partial(e^{x,y}) + \mathcal N = \tilde t_{x,y}.$ Hence, $\tilde t_{x,y}\in \mathcal J/\cN.$ 

Now let us prove that $\mathcal J/\cN$ is generated by $\tilde t_{x,y}.$ In this proof we say that a tuple $(x_0,\dots,x_n)$ is an \textit{almost path} if there is $0\leq k<n$ such that $d(x_k,x_{k+1})=2$ and $d(x_j,x_{j+1})=1$ for $j\neq k.$ It is easy to check that, if $(x_0,\dots,x_n)$ is not a path and not an almost path, then $\partial(e^{x_0,\dots,x_n})\in \mathcal N.$ If $(x_0,\dots,x_n)$ is an almost path, then
\begin{equation}
\partial(e^{x_0,\dots,x_n})+\mathcal N 
= 
e^{x_0,\dots,x_{k-1}} \tilde t_{x_k,x_{k+1}} e^{x_{k+2},\dots,x_n}. \end{equation}
The result follows. 
\end{proof}

\begin{theorem}\label{theorem:omega} Let $G$ be a finite digraph and $\KK$ be a commutative ring. Then
there is an isomorphism of dg-algebras 
\begin{equation}
\Omega^\bullet(G) \cong \KK G/T,
\end{equation}
sending $e^{x_0,\dots,x_n}+\mathcal{J}$ to $(x_0,\dots,x_n)+T,$ where $T$ is the ideal generated by quadratic relations $t_{x,y}$ which are indexed by pairs of vertices $x,y$ at a distance $d(x,y)=2,$ and defined by the formula
\begin{equation}\label{eq:t}
t_{x,y}= \sum_{v} (x,v,y), 
\end{equation}
where the sum runs over all vertices $v$ such that $(x,v,y)$ is a path. 
The differential on $\KK G/T$ is defined on a path $p$ by the formula  
\begin{equation}
\partial (p) = \sum_{i=0}^{n+1}\  \sum_{v} (-1)^i (p_0,\dots,p_{i-1},v,p_{i},\dots,p_n),
\end{equation}
where the second sum runs over all vertices $v$ such that $(p_0,\dots,p_{i-1},v,p_{i},\dots,p_n)$ is a path in $\KK G.$
\end{theorem}
\begin{proof}
It follows from Lemmas \ref{lemma:iso:R(X)},  \ref{lemma:iso:R/N} and  \ref{lemma:generators_of_J}. 
\end{proof}

Using this Theorem we can obtain the dual result to the result of Asao \cite[Lemma 6.8]{asao2023magnitude}. 

\begin{corollary}\label{cor:diagonal_magnitude_and_path} For any commutative ring $\KK$ and any finite digraph $G,$ there is an isomorphism of graded algebras of the path cochain algebra and the diagonal part of the magnitude cohomology algebra
\begin{equation}
\Omega^\bullet(G)\cong \MH^{\sf diag}(G)
\end{equation}
\end{corollary}
\begin{proof}
It follows from Theorem \ref{theorem:omega} and \cite[Th.6.2]{hepworth2022magnitude} (the proof of Theorem \cite[Th.6.2]{hepworth2022magnitude} is written only in the case $\KK=\ZZ$ but it can be generalized to any commutative ring without any changes).
\end{proof}

\section{Diagonal digraphs and Koszul algebras}

In this section we show that a finite  digraph $G$ is diagonal if and only if the distance algebra $\sigma G$ is Koszul for any field, and if and only if $G$ satisfies $(\mathcal{V}_2)$ and the path cochain algebra $\Omega^\bullet(G)$ is Koszul for any field. Moreover, for a digraph $G$ satisfying $(\mathcal{V}_2)$, we give an explicit description of the Koszul complex of the quadratic algebra $\Omega^\bullet(G)$  (Proposition~\ref{prop:koszul_complex}). This provides another equivalent description of diagonal digraphs from a completely different viewpoint (Corollary  \ref{cor:diagonality:koszul_complex}).

\subsection{Background on quadratic and Koszul algebras} 

Here we will recall the basic information about quadratic and Koszul algebras in the setting of graded quiver algebras over fields \cite{martinez2007introduction}, \cite{beilinson1996koszul}.

Let $A=\bigoplus_{n\geq 0} A_n$ be a non-negatively graded ring and $S=A_0$. If $M,N$ are graded $A$-modules we denote by $\Ext^{n,\ell}_A(M,N)$ the bigraded Ext-functor in the category of graded modules. If $M,N$ are ungraded $A$-modules, we set $\Ext^n_A(M,N)$ the ordinary Ext functor. The graded ring $A$ is called Koszul if $S$ is a semisimple ring and there exists a graded projective resolution of $S$ over $A$
\begin{equation}
  \dots  \to  P_1 \to P_0 \to  S \to 0
\end{equation}
such that $P_n$ is a graded projective $A$-module generated in degree $n.$ It is known \cite[Prop.2.1.3]{beilinson1996koszul} that 
\begin{equation}\label{eq:Koszul}
A \text{ is Koszul} \hspace{5mm}    \Leftrightarrow \hspace{5mm} \Ext^{n,\ell}_A(S,S)=0 \text{ for } \ell\neq n.
\end{equation}
We denote by 
\begin{equation}
E(A) = \Ext^*_A(S,S)
\end{equation}
the Ext algebra, where the product is given by the Yoneda product. We will restrict ourselves to graded quiver algebras over a field.

Let $\KK$ be a field, $Q$ be a finite quiver, and $\KK Q=\bigoplus_{n\geq 0} (\KK Q)_n$ be the graded path algebra, where the grading is defined by the lengths of paths. For a homogeneous ideal $I$ of $\KK Q$ we denote by $I_n$ its homogeneous components. Then $I$ is called admissible if $I_0=I_1=0.$ By a \emph{graded quiver algebra} we mean an algebra of the form $A=\KK Q/I,$ where $I$ is a graded admissible ideal. 

A graded quiver algebra $A=\KK Q/I$ is called \emph{quadratic} if $I$ is generated by $I_2.$ It is known \cite[Theorem 2.3.2]{beilinson1996koszul}, for a graded quiver algebra $A,$ we have
\begin{equation}\label{eq:quadratic}
A \text{ is quadratic} \hspace{5mm}    \Leftrightarrow \hspace{5mm} \Ext^{2,\ell}_A(S,S)=0 \text{ for } \ell\neq 2.
\end{equation}

Assume that $A=\KK Q/I$ is a quadratic algebra. Denote by $Q^{op}$ the quiver opposite to $Q$ and, for any path $p$ in $Q,$ denote by  $p^{op}$ the opposite path in $Q^{op}.$ Then there is a non-degenerate bilinear form 
\begin{equation}
\langle -,=  \rangle:(\KK Q)_2 \times (\KK Q^{op})_2 \to \KK    
\end{equation}
which is defined on paths via Kronecker delta $\langle p,q \rangle=\delta_{p^{op},q}.$  For a quadratic algebra $A=\KK Q/I$ we  
denote by $I_2^{\bot}$ the vector subspace of $(\KK Q^{op})_2$ orthogonal to $I_2$ with respect to this bilinear form. We also denote by $I^\bot$ the ideal of $\KK Q^{op}$ generated by $I_2^\bot.$ Then the dual quadratic algebra to $A$ is defined by 
\begin{equation}
A^! = (\KK Q^{op})/I^\bot.
\end{equation}

If a graded quiver algebra is Koszul, it is quadratic, and a quadratic algebra $A$ is Koszul if and only if $A^!$ is Koszul. Moreover, for any Koszul graded quiver algebra $A$ there is an isomorphism
\begin{equation}
E(A)\cong A^!.
\end{equation}

\subsection{Diagonal digraphs}

\begin{lemma}\label{lemma:diagonal_free}
For any digraph $G$ and any $n$ the group $\MH_{n,n}(G,\ZZ)$ is free abelian. 
\end{lemma}
\begin{proof}
Since $\MC_{n+1,n}(G,\ZZ)=0,$ we have $\MH_{n,n}(G,\ZZ)\subseteq \MC_{n,n}(G,\ZZ),$ and we use the fact that a subgroup of a free abelian group is free abelian. 
\end{proof}

\begin{lemma}\label{lemma:universal_coefficient} For any digraph $G$ and any $n,\ell$ there are short exact sequences
\begin{equation}\label{eq:universal_coefficient}
  \MH_{n,\ell}(G,\ZZ)\otimes_\ZZ \KK \mono  
  \MH_{n,\ell}(G,\KK) 
  \epi \Tor^\ZZ_1(\MH_{n-1,\ell}(G,\ZZ),\KK), 
\end{equation} 
\begin{equation}
    {\rm Ext}_\ZZ^1(\MH_{n-1,\ell}(G,\ZZ),\KK ) \mono \MH^{n,\ell}(G,\KK) \epi \Hom_\ZZ( \MH_{n,\ell}(G,\ZZ),\KK).
\end{equation}    
\end{lemma}
\begin{proof}
It follows from the universal coefficient theorem for chain complexes \cite[Th.3.6.1, Th.3.6.5]{weibel1994introduction},  the isomorphisms 
\begin{equation}
 \MC_{\bullet,\ell}(G,\KK)=\MC_{\bullet,\ell}(G,\ZZ)\otimes_\ZZ \KK, \hspace{5mm} \MC^{\bullet,\ell}(G,\KK)\cong \Hom_\ZZ(\MC_{\bullet,\ell}(G,\ZZ),\KK) 
\end{equation}
 and the fact that the components of $\MC_{\bullet,\ell}(G,\ZZ)$ are free abelian groups.
\end{proof}

A digraph $G$ is called diagonal if $\MH_{n,\ell}(G,\ZZ)=0$ for any $\ell\neq n$ \cite{hepworth2017categorifying}.

\begin{lemma}
\label{lemma:diagonal} 
For a finite digraph $G$ the following conditions are equivalent:
\begin{enumerate}
\item $G$ is diagonal;
\item $G^{op}$ is diagonal;
\item $\MH_{n,\ell}(G,\KK)=0$ for $\ell\neq n$ and any commutative ring $\KK;$
\item $\MH_{n,\ell}(G,\KK)=0$ for $\ell\neq n$ and any prime field $\KK;$
\item $\MH^{n,\ell}(G,\ZZ)=0$ for $\ell\neq n;$
\item $\MH^{n,\ell}(G,\KK)=0$ for $\ell\neq n$ and any commutative ring $\KK;$
\item $\MH^{n,\ell}(G,\KK)=0$ for $\ell\neq n$ and any prime field $\KK.$
\end{enumerate}
\end{lemma}
\begin{proof}
It is easy to see that there is an isomorphism  $\MC_{\bullet,\ell}(G,\KK)\cong \MC_{\bullet,\ell}(G^{op},\KK)$ sending $(x_0,\dots,x_n)\mapsto (x_n,\dots,x_0).$ Hence  $(1) \Leftrightarrow (2).$ 
If $G$ is diagonal, then by Lemma \eqref{lemma:diagonal_free} all its magnitude homology groups over $\ZZ$ are free. Therefore, using Lemma \ref{lemma:universal_coefficient}  we obtain that  $(1)$ implies $(3),$ $ (4),$ $(5),$ $(6),$ $(7).$ Obviously $(3)\Rightarrow (4)$ and $(6) \Rightarrow (7).$ Since, for any field $\KK$ we have an isomorphism $\MH^{n,\ell}(G,\KK)\cong \Hom_{\KK}(\MH_{n,\ell}(G,\KK),\KK),$ we obtain $(4)\Leftrightarrow (7).$ 
Therefore, it is sufficient to prove that $(4)\Rightarrow (1).$

$(4)\Rightarrow (1).$ In general, if $A$ is a finitely generated abelian group such that $A\otimes_\ZZ \KK=0$ for any prime field $\KK$, then $A=0.$ Since $G$ is finite, the abelian groups $\MH_{n,\ell}(G,\ZZ)$ are finitely generated. Hence the equation $\MH_{n,\ell}(G,\KK)=0$ for any $\ell\neq n$ and a prime field $\KK$ combined with the short exact sequence \eqref{eq:universal_coefficient} implies that $\MH_{n,\ell}(G,\ZZ)\otimes_\ZZ \KK=0$ for any $\ell\neq n$ and any field $\KK.$ Therefore, $\MH_{n,\ell}(G,\ZZ) =0 $ for any $n\neq \ell.$  
\end{proof}

\begin{proposition}\label{prop:V_2}
The following statements about a finite digraph $G$ are equivalent.
\begin{enumerate}
\item $\MH_{2,\ell}(G,\ZZ)=0$ for $\ell\neq 2;$
\item $G$ satisfies $(\mathcal{V}_2);$
\item $\sigma G$ is quadratic for any field $\KK.$
\end{enumerate}
Moreover, if these properties are satisfied, then 
\begin{equation}\label{eq:omega_is_dual}
  (\sigma G)^! \cong \Omega^\bullet(G^{op})
\end{equation}
for any field $\KK.$
\end{proposition}
\begin{proof}
The equivalence of (1) and (2) follows from Theorem \ref{th:vanishing}.
Using Lemma \ref{lemma:universal_coefficient} and the fact that $\MH_{1,\ell}(G,\ZZ)$ is free abelian, we obtain that $(1)$ is equivalent to the fact that $\MH^{2,\ell}(G,\KK)=0$ for $\ell\neq 2$ and any field $\KK.$ Then by Theorem \ref{theorem:magnitude_as_derived} and \eqref{eq:quadratic} we obtain that (1) is equivalent to (3). 

Let us assume that $\sigma G$ is quadratic and prove \eqref{eq:omega_is_dual}. Since $\sigma G=\KK G/R$ is quadratic, $R$ is generated by $R_2.$ Denote by $R_2^{\tS}$ the vector space generated by differences of shortest paths of length $2$ connecting the same vertices, and by $R_2^{\tL}$ the vector space generated by long paths of length two. Then $R_2=R_2^{\tS} \oplus R_2^{\tL}.$ Then $\dim R_2^{\tL}$ is equal to the number of long paths of length $2,$ and $\dim R_2^{\tS}$ is equal to the number of short paths of length two minus the number of pairs at distance two. Therefore $\dim (\KK G)_2 - \dim R_2$ is equal to the number of pairs at distance two. It follows that $\dim R_2^\bot$ is equal to the number of pairs at distance two.  On the other hand, it is easy to check that any element  of the form $t_{y,x} = \sum_{v\in ]y,x[} (y,v,x)  \in (\KK G^{op})_2$ (see \eqref{eq:t}) is orthogonal to any difference of shortest paths $(x,v_1,y)-(x,v_2,y)$ of length $2$ connecting the same vertices, and to any long path of length $2.$ It follows that $R^\bot_2$ is generated by the elements of the form $t_{y,x}\in (\KK G^{op})_2.$ 
\end{proof}

\begin{theorem}\label{th:Koszul}  The following statements about a finite digraph $G$ are equivalent:
\begin{enumerate}
\item $G$ is diagonal;
\item  $\sigma G$ is Koszul for any field $\KK;$
\item $G$ satisfies $(\mathcal{V}_2)$ and $\Omega^\bullet(G)$ is Koszul for any field $\KK.$
\end{enumerate}
\end{theorem}
\begin{proof}
The equivalence $(1)\Leftrightarrow (2)$ follows from the isomorphism $\MH^{n,\ell}(G,\KK)\cong  \Ext^{n,\ell}_{\sigma G}(S,S)$ (Theorem \ref{theorem:magnitude_as_derived}), the equivalence \eqref{eq:Koszul}, and Lemma \ref{lemma:diagonal}. 
The implications $(1)\& (2) \Rightarrow (3)$ and $(3) \Rightarrow (2)$ follow from Proposition \ref{prop:V_2} and the fact that the quadratic dual of a Koszul algebra is Koszul. 
\end{proof}

\subsection{Koszul complex of a digraph}

Let $\KK$ be a field, $A=\KK Q/I$ be a quadratic algebra, $A^!=\KK Q^{op}/I^\bot$ be its dual and $S=(\KK Q)_0=(\KK Q^{op})_0$. We treat the $n$-th homogeneous components $A^!_n$ as a right $S$-module, and  $A$ as a $(A,S)$-bimodule, and consider a chain complex ${\sf K}_\bullet$ of graded  left $A$-modules, whose components are described as graded hom-modules of right $S$-modules
\begin{equation}
{\sf K}_n=\Hom_S(A^!_n,A)\cong A\otimes_S \Hom_S(A^!_n,S) 
\end{equation}
and the differential $\partial:{\sf K}_n\to {\sf K}_{n-1}$ is defined by 
$\partial( \varphi )(a) = \sum_{\alpha\in Q_1} \varphi(a\alpha^{op})\cdot \alpha.$ Here we assume that ${\sf K}_n$ is a graded module, whose homogeneous components are 
\begin{equation}
({\sf K}_n)_{n+i}=\Hom_S(A^!_n,A_i) \cong A_i\otimes_S \Hom_S(A^!_n,S).
\end{equation}
The chain complex ${\sf K}_\bullet$ is called 
\emph{the Koszul complex} of $A.$ 
Consider a map $\varepsilon:{\sf K}_0\to S$ 
defined by $\varepsilon(\varphi)=\varphi(1).$ Then it is known that \cite[\S 2.6,\S 2.8]{beilinson1996koszul}
\begin{equation}\label{eq:koszul_complex}
    A \text{ is Koszul } \hspace{5mm}\Leftrightarrow \hspace{5mm} {\sf K}_\bullet \text{ is a resolution of } S.
\end{equation}

Now let $G$ be a digraph and $\KK$ be a field. Consider a graded chain complex ${\sf K}^G_\bullet$ of left $\Omega^\bullet(G)$-modules whose components are defined by 
\begin{equation}
{\sf K}^G_n = \bigoplus_{d(x,y)=n} \Omega^\bullet(G) \cdot e_x, 
\end{equation}
where the sum is taken over all pairs of vertices $(x,y)$ such that $d(x,y)=n.$ Denote by $\kappa_{x,y} \in {\sf K}^G_n$ the image of $e_x$ in the summand indexed by $(x,y).$  Then the grading is defined so that $|\kappa_{x,y}|=d(x,y)$ and the differential $\partial:{\sf K}^G_n\to {\sf K}^G_{n-1}$ is defined on the generators by
\begin{equation}
\partial(\kappa_{x,y}) = \sum_{ \substack{ v : d(x,v)=1\\ d(v,y)=n-1}}  (x,v) \cdot  \kappa_{v,y}, 
\end{equation}
where the sum is taken over all vertices $v$ such that $d(x,v)=1$ and $d(v,y)=n-1.$ 

\begin{proposition}\label{prop:koszul_complex} Let $G$ be a finite digraph satisfying $(\mathcal V_2).$ Then the chain complex ${\sf K}^G_\bullet$ is well defined and it is isomorphic to the Koszul complex of the path cochain algebra $\Omega^\bullet(G).$    
\end{proposition}
\begin{proof} Denote by ${\sf K}_\bullet$ the Koszul complex of $\Omega^\bullet(G).$
By Proposition \ref{prop:V_2} we have isomorphisms
\[(\Omega^\bullet(G))^! \cong \sigma G^{op}, \hspace{1cm} (\sigma G^{op})_n = \bigoplus_{d(x,y)=n} \KK \cdot (y,x)\] and \[\Hom_{S}(\KK \cdot (y,x),\Omega^\bullet(G)) \cong  \Omega^\bullet(G) \cdot e_x.\] 
Therefore, we obtain 
${\sf K}_n\cong \Hom_S( (\sigma G^{op})_n , \Omega^\bullet(G))\cong {\sf K}^G_n.$ For each pair of vertices $x,y$ such that $d(x,y)=n,$ we consider the elements $\theta_{x,y}\in {\sf K}_n$ defined by  
 $\theta_{x,y}(b,a)=\delta_{(x,y),(a,b)}e_x$ and $\delta$ is the Kronecker delta. Then the isomorphism $\theta:{\sf K}^G_n\to C_n$ is defined by $\theta(\kappa_{x,y})=\theta_{x,y}.$ For any pair of vertices $(a,b)$ such that $d(a,b)=n-1$ a computation shows that 
 \begin{equation}
 \begin{split}
    \partial(\theta_{x,y})(b,a) &= \sum_{ d(s,t)=1} \theta_{x,y}((b,a)(t,s)) \cdot  (s,t) \\ 
    & = \delta_{y,b} \cdot \delta_{d(x,a),1}   \cdot (x,a)
\end{split}
 \end{equation}
 and 
 \begin{equation}
 \begin{split}
 \sum_{ \substack{v: d(x,v)=1\\ d(v,y)=n-1}}  (x,v) \cdot  \theta_{v,y}(b,a) & = \sum_{ \substack{v: d(x,v)=1\\ d(v,y)=n-1}}  (x,v) \cdot  \delta_{(v,y),(a,b)} \\
 &= \delta_{y,b} \cdot  \delta_{d(x,a),1} \cdot  (x,a). 
 \end{split}
 \end{equation}
Therefore we obtain 
\begin{equation}
\partial(\theta_{x,y}) =  \sum_{ \substack{v: d(x,v)=1\\ d(v,y)=n-1}}  (x,v) \cdot  \theta_{v,y}. 
\end{equation}
This implies that the differential on ${\sf K}^G_\bullet$ is well defined and $\theta$ is an isomorphism of chain complexes. 
\end{proof}

\begin{corollary}\label{cor:diagonality:koszul_complex}
A finite digraph $G$ is diagonal if and only if it satisfies $(\mathcal V_2)$ and ${\sf K}^G_\bullet$ is a resolution of $S$ for any field $\KK.$
\end{corollary}
\begin{proof}
It follows from Theorem \ref{th:Koszul}, Proposition \ref{prop:koszul_complex} and \eqref{eq:koszul_complex}.
\end{proof}

\begin{remark}
Note that the chain complex ${\sf K}^G_\bullet$ can be decomposed into a  direct sum of subcomplexes indexed by vertices
\begin{equation}
    {\sf K}^G_\bullet = \bigoplus_{y} {\sf K}^{G,y}_\bullet, 
\end{equation}
where ${\sf K}^{G,y}_n$ is generated by $\kappa_{x,y}$ for all $x$ such that $d(x,y)=n.$ Moreover, ${\sf K}^G_\bullet$ is a resolution of $S$ if and only if ${\sf K}^{G,y}_\bullet$ is a resolution of $S_y$ for each $y.$
\end{remark}

\section{Extended Hasse diagrams}

In this section, we develop the ideas of Kaneta--Yoshinaga \cite[\S 5.3]{kaneta2021magnitude} in the context of directed graphs. First, we study the magnitude homology of a digraph $G_P$ associated with a ranked poset $P$. We describe the magnitude homology of $G_P$ in terms of the homology of posets. Next, we apply this to describe the homology of a digraph $\hat{G}_K$ associated with a pure simplicial complex $K$, which we call the extended Hasse diagram. Finally, we show that if $K$ is a triangulation of a manifold, then the non-diagonal part of the magnitude homology of $\hat{G}_K$ can be expressed in terms of the homology of the manifold. As a corollary, we deduce that $\hat{G}_K$ is diagonal if and only if the manifold is a homology sphere.

\subsection{Ranked posets}
For a poset $P$ and a commutative ring $\KK$ we denote by $C_\bullet(P)$ the chain complex whose components are $C_n(P)=\KK\cdot \{(x_0,\dots,x_n)\mid x_0<\dots<x_n\}$ and the differential is defined as the alternating sum of deleting maps $\partial(x_0,\dots,x_n) = \sum_{i=0}^n (-1)^i(x_0,\dots, \hat x_i,\dots,x_n)$. We denote by $\bar C_\bullet(P)$ a chain complex, whose non-negative part coincides with $C_\bullet(P),$ and that has one more non-zero component $\bar C_{-1}(P)=\KK.$ The differential $\bar C_0(P)\to \bar C_{-1}(P)$ sends $(x_0)$ to $1.$ Then the reduced homology of $P$ with coefficients in $\KK$  is defined by the formula
\begin{equation}
\bar H_n(P) = H_n(\bar C_\bullet(P)). 
\end{equation}
If we need to specify $\KK,$ we will denote $\bar H_*(P)$ by $\bar H_*(P,\KK).$ Note that 
\begin{equation}
\bar  H_{-1}(\emptyset) = \KK.    
\end{equation}

A \emph{chain} in a poset $P$ is a totally ordered subset. An element $y\in P$ \emph{covers} an element $x\in P,$ if $x<y$ and there is no $z$ such that $x<z<y.$ A \emph{ranked poset} is a poset $P$ where every maximal chain is finite and all maximal chains have the same number of elements. Then there is a uniquely defined rank function $r:P\to \mathbb N$ such that $r(x)=0$ if and only if $x$ is minimal, and if $y$ covers $x,$ then $r(y)=r(x)+1.$ In this case all maximal elements of $P$ have the same rank, which we call the dimension of $P,$ and denote by $\dim P.$ The dimension of the empty set is defined as $-1.$ Note that for any elements $x\leq y$ of $P$ the open interval 
\begin{equation}
]x,y[ \: = \{ z\in P\mid x<z<y\}  
\end{equation}
is also a ranked poset of dimension $r(y)-r(x)-1.$

If $P$ is a ranked poset, we denote by $G_P$ the digraph, whose vertices are elements of $P$ and there is an arrow  $x\to y$ if and only if $y$ covers $x.$ Then the distance $d(x,y)$ is finite if and only if $x\leq y$ and in this case $d(x,y)=r(y)-r(x).$

\begin{lemma}\label{lemma:ranked_poset}
Let $P$ be a ranked poset and $x,y$ be its elements. If $x < y,$ then there is an isomorphism of chain complexes 
\begin{equation}
\MC_{\bullet,r(x)-r(y)}(G_P)_{x,y} \cong \bar C_{\bullet-2}(\:]x,y[\:).
\end{equation}
Moreover, $\MC_{\bullet,\ell}(G(P))_{x,y}=0$ if  $\ell\neq r(x)-r(y)$ or $x\not\leq y.$
\end{lemma}
\begin{proof}
If $x\not\leq y,$ then $d(x,y)=\infty,$ and hence $\MC_{\bullet,\ell}(G_P)_{x,y}=0.$  Since $|x_0,\dots,x_n|$ is equal to either  $r(x_n)-r(x_0)$ or to infinity, if $\ell\neq r(x)-r(y),$ then $\MC_{\bullet,\ell}(G_P)_{x,y}=0.$ 
Now assume that $x<y$ and $\ell=r(y)-r(x).$ 
If $n\neq 1,$ then the map $(x_0,\dots,x_{n})\mapsto (x_1,\dots,x_{n-1})$ defines a bijection from the set
\begin{equation}
\{ (x_0,\dots,x_{n})\mid x_i\neq x_{i+1}, x_0=x,x_{n}=y, |x_0,\dots,x_{n}|=\ell \}
\end{equation}
to the set 
\begin{equation}
\{ (x_1,\dots,x_{n-1})\mid x<x_1<\dots<x_{n-1}<y\}.    
\end{equation}
This bijection defines an isomorphism $f_n:\MC_{n,\ell}(G_P)_{x,y} \overset{\cong}\to  \bar C_{n-2}(\:]x,y[\:)$ for $n\neq 1.$ If $n=1,$ we have $\MC_{1,\ell}(G_P)_{x,y}\cong \langle (x,y) \rangle \cong \KK$ and denote by $f_1:\MC_{1,\ell}(G_P)_{x,y} \to \KK$ this isomorphism. It is easy to see that $f_n$ is compatible with the differential up to sign. It follows that the maps $(-1)^n f_n$ define an isomorphism of chain complexes. 
\end{proof}

\begin{proposition}\label{prop:ranked_poset}
Let $P$ be a ranked poset, $x,y$ be its elements and $\ell\geq 1$. Then 
\begin{equation}
\MH_{n,\ell}(G_P)_{x,y} \cong
\begin{cases}
\bar H_{n-2}(\:]x,y[\:), & \ell=r(x)-r(y),\  x<y \\
0, & \text{otherwise.} 
\end{cases}
\end{equation}
In particular, for $\ell\geq 1,$ we obtain 
\begin{equation}\label{eq:ranked_poset}
\MH_{n,\ell}(G_P) \cong \bigoplus_{x<y,  r(y)-r(x)=\ell} \bar H_{n-2}(\:]x,y[\:).
\end{equation}
\end{proposition}
\begin{proof}
It follows from Lemma \ref{lemma:ranked_poset}.
\end{proof}

\subsection{Pure simplicial complexes}

Let $K$ be a pure (abstract) simplicial complex of dimension $d$.  The link of a simplex $\sigma\in K$ is a simplicial subcomplex of $K$ defined by 
\begin{equation}
\Lk(\sigma)=\{\tau\in K \mid \tau\cap \sigma =\emptyset, \tau\cup \sigma\in K\}.
\end{equation}
The link $\Lk(\sigma)$ is also a pure simplicial complex of dimension $d-\dim(\sigma)-1.$

We denote by $P(K)$ the face poset of $K,$ defined by $P(K)=(K,\subseteq).$ It is well known that the homology of $K$ coincides with $P(K)$
\begin{equation}
H_*(K) \cong  H_*(P(K)),
\end{equation}
because the order complex of $P(K)$ is the barycentric subdivision of $K$, and the  homology of a poset $P$ is the homology of its order complex \cite[\S 1.1]{wachs2007poset}. We denote by $L(K)$ the face lattice of $K$, defined by 
\begin{equation}
L(K):=K\sqcup \{\hat 0, \hat 1\},    
\end{equation} 
where the simplices are ordered by inclusion and $\hat 0<\sigma<\hat 1$ for any $\sigma\in K$ \cite[\S 1.1]{wachs2007poset}. Then $L(K)$ is a ranked poset of dimension $d+2.$ The \emph{extended Hasse diagram} of $K$ is defined by 
\begin{equation}
\hat G_K=G_{L(K)}. 
\end{equation}
We also denote by $K(n)$ the set of $n$-simplices of $K.$

\begin{theorem}\label{th:simplicial_complex}
For any pure simplicial complex $K$ of dimension $d$ there are isomorphisms
\begin{align}
\MH_{n,d+2}(\hat G_K)_{\hat 0, \hat 1} &\cong  \bar H_{n-2}(K), \label{eq:hasse_simplicial_1} \\
\MH_{n,\ell}(\hat G_K)_{\sigma, \hat 1 } & \cong  \bar H_{n-2}(\Lk(\sigma)),  &\ell = d+1- \dim(\sigma) \\
\MH_{n,n}(\hat G_K)_{\hat 0,\sigma} &\cong \KK,  &n=\dim(\sigma)+1, \\
\MH_{n,n}(\hat G_K)_{\sigma,\tau}&\cong \KK, & n=\dim(\tau)-\dim(\sigma), \sigma\subsetneq \tau
\end{align}
and all other modules $\MH_{n,\ell}(\hat G_K)_{x,y}$ for $\ell\geq 1$ are trivial. Moreover, for $\ell\geq 1,$ we obtain 
\begin{equation}
\MH_{n,\ell}(\hat G_K) \cong 
\begin{cases}
\bar H_{n-2}(K), & \ell=d+2\\
\bigoplus_{\sigma \in K(d+1-\ell)} \bar H_{n-2}(\Lk(\sigma)), & 1\leq \ell\neq d+2,\  n \neq \ell\\
\KK^{D(n)}\oplus \left(\bigoplus_{\sigma \in K(d+1-n)} \bar H_{n-2}(\Lk(\sigma))\right), & 1\leq n=\ell \neq d+2,\\
\end{cases}
\end{equation}
where
$D(n) = \sum_{i=0}^{d-n+1} \binom{n+i}{i} \cdot  |K(n+i-1)|.$
\end{theorem}
\begin{proof}
By Proposition \ref{prop:ranked_poset}, we know that $\MH_{n,\ell}(\hat G_K)_{x,y}$  is isomorphic to $H_{n-2}(\:]x,y[\:),$ if $x<y$ and $\ell=r(x)-r(y),$ and equal to zero otherwise. Here we denote by $r$ the rank function for $L(K),$ which is defined by $r(\hat  0)=0,r( \hat 1)=d+2$ and $r(\sigma)=\dim(\sigma)+1.$

Assume that $(x,y)=(\hat 0, \hat  1).$ Then we obtain that $\MH_{n,\ell}(\hat G_K)_{\hat  0, \hat 1}$ is trivial for $\ell\neq d+2,$ and $\MH_{n,d+2}(\hat G_K)_{ \hat 0, \hat 1}=\bar H_{n-2}(\:]\hat 0, \hat 1[\:)\cong \bar H_{n-2}(P(K)) \cong  \bar H_{n-2}(K).$

Assume that $(x,y)=(\sigma, \hat 1 ).$ Then $\MH_{n,\ell}(\hat G_K)_{\sigma, \hat 1}$ is trivial for $\ell\neq d+1-\dim(\sigma).$ Note that there is an isomorphism of posets $\Lk(\sigma) \cong \: ]\sigma, \hat 1 [\:$ defined by $\tau \mapsto \sigma\cup \tau.$ Then $\MH_{n,\ell}(\hat G_K)_{\sigma, \hat  1}\cong \bar H_{n-2}(\Lk(\sigma))$ for $\ell=d+1-\dim(\sigma).$

Assume that $(x,y)=( \hat  0,\sigma).$ Then $\MH_{n,\ell}(\hat G)_{ \hat 0,\sigma}$ is trivial for $\ell\neq \dim(\sigma)+1.$ Note that there is an isomorphism $\:] \hat 0,\sigma[\: = \{\tau\in K\mid \tau\subsetneq \sigma\} \cong \partial \Delta^{\dim(\sigma)}.$ Hence $\MH_{n,\ell}(\hat G_K)_{ \hat 0,\sigma} \cong \bar H_{n-2}(\partial \Delta^{\dim(\sigma)})$ for $\ell=\dim(\sigma)+1.$ Hence $\MH_{n,n}(\hat G_K)_{ \hat 0,\sigma}\cong \KK$ for $n=\dim(\sigma)+1,$ and $\MH_{n,\ell}(\hat G_K)_{ \hat 0,\sigma}$ is trivial otherwise.

Assume that $(x,y)=(\sigma,\tau),$ where $\sigma\subsetneq\tau.$ Then $\MH_{n,\ell}(\hat G_K)_{\sigma,\tau}=0$ for $\ell\neq \dim(\tau)-\dim(\sigma).$
The isomorphism $\Lk(\sigma)\cong \:]\sigma,1[\:,$ $\theta\mapsto \theta\cup \sigma$ restricts to an isomorphism $\{ \theta\mid \emptyset \neq \theta \subsetneq \tau\setminus \sigma \} \cong \:]\sigma,\tau[\:.$ Therefore $\:]\sigma,\tau[\:\cong \partial\Delta^{\dim(\tau)-\dim(\sigma)-1}.$ It follows that $\MH_{n,n}(\hat G_K)_{\sigma,\tau}\cong \KK$ for $n=\dim(\tau)-\dim(\sigma),$ and $\MH_{n,\ell}(\hat G_K)_{\sigma,\tau}$ is trivial otherwise. 

The formula for $D(n)$ follows from the facts  that there are $|K(n-1)|$ simplices of dimension $n-1$ and for $i\geq 1$ we have $\binom{n+i}{i}$ simplices $\sigma\subseteq \tau$ of dimension $i-1$ for each simplex $\sigma\in K(n+i-1).$
\end{proof}

Let $K$ be a pure simplicial complex of dimension $d.$ We say that the reduced homology of $K$ is \emph{concentrated in the top dimension} if $\bar H_i(K,\ZZ)=0$ for $i\neq d.$ In this case $H_d(K,\ZZ)$ is a free abelian group and the universal coefficient theorem implies that  $\bar H_{i}(K,\KK)=0$ for $i\neq d$ for any commutative ring $\KK.$ Note that we do not assume that $H_d(K,\ZZ)\neq 0.$

\begin{proposition}\label{proposition:triv_reduced}
Let $K$ be a pure simplicial complex such that, for any $\sigma\in K,$ the reduced homology of $\Lk(\sigma)$ is concentrated in the top dimension. Then for $n\neq \ell$ we have 
\begin{equation}
\MH_{n,\ell}(\hat G_K) \cong 
\begin{cases}
\bar H_{n-2}(K),&    \ell-2=\dim(K),\\
0,& \text{otherwise.}
\end{cases}
\end{equation}
\end{proposition}
\begin{proof}
By Theorem \ref{th:simplicial_complex} we obtain that it is sufficient to prove that $\bar H_{n-2}(\Lk(\sigma))=0$ for any $\sigma$ such that $\ell\neq n,$ where $\ell=\dim(K)+1-\dim(\sigma).$ Since $\dim(\Lk(\sigma))=\dim(K)-\dim(\sigma)-1,$ then 
$\ell\neq n$ is equivalent to $\dim(K)\neq n-2.$ This proves the claim.
\end{proof}

\begin{proposition}
Let $K$ be a pure simplicial complex. Then $\hat G_K$ is diagonal if and only if 
\begin{enumerate}
    \item the reduced homology of $K$ is concentrated in the top dimension;
    \item the reduced homology of $\Lk(\sigma)$ is concentrated in the top dimension for any $\sigma\in K.$
\end{enumerate}  
\end{proposition}
\begin{proof} Assume that $\KK=\ZZ$ and set $d=\dim(K).$ Theorem \ref{th:simplicial_complex} implies that $\hat G_K$ is diagonal if and only if the following two properties are satisfied:
(1) $\bar H_{n-2}(K)=0$ for $n\neq d+2;$ (2) $\bar H_{n-2}(\Lk(\sigma))=0$ for $n\neq d+1-\dim(\sigma).$  The first property is equivalent to $\bar H_{m}(K)=0$ for $ m\neq d.$ The second property is equivalent to $\bar H_{m}(\Lk(\sigma))=0$ for $m\neq d-1 - \dim(\sigma)=\dim(\Lk(\sigma)).$
\end{proof}

\begin{proposition}
Let $K$ be a connected pure simplicial complex of dimension at least two and $\ell\geq 2$. Then $\hat G_K$ satisfies property $(\mathcal V_\ell)$ if and only if $\Lk(\sigma)$ is connected for any $\sigma\in K$ of dimension $\dim(\sigma)\leq \dim(K)-\ell.$ 
\end{proposition}
\begin{proof}
Assume that $\KK=\ZZ$ and set $d=\dim(K)\geq 2.$ Theorem \ref{th:simplicial_complex} implies that $\hat G_K$ is diagonal if and only if  $\bar H_{0}(\Lk(\sigma))=0$ for $k\geq \ell+1,$ where $k=d+1-\dim(\sigma).$ The assertion follows. 
\end{proof}

\begin{example}\label{example:triangle}
In this example we consider the following simplicial complex $K,$ 
\def\constA{0.3}
\def\constB{1.3}
\def\aA{({\constA*cos(90-0*360/3)},{\constA*sin(90-0*360/3)})}
\def\aB{({\constA*cos(90-1*360/3)},{\constA*sin(90-1*360/3)})}
\def\aC{({\constA*cos(90-2*360/3)},{\constA*sin(90-2*360/3)})}
\def\bA{({\constB*cos(90-0*360/3)},{\constB*sin(90-0*360/3)})}
\def\bB{({\constB*cos(90-1*360/3)},{\constB*sin(90-1*360/3)})}
\def\bC{({\constB*cos(90-2*360/3)},{\constB*sin(90-2*360/3)})}
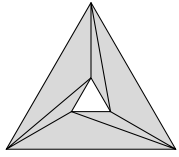
\[
\begin{tikzpicture}
\draw [color=white, fill=gray!30] plot  coordinates {
\bA \bB \bC 
};
\draw [color=white, fill=white] plot  coordinates {
\aA \aB \aC 
};
\draw[line width=1pt]
\aA  node (a0) {}
\aB  node (a1) {}
\aC  node (a2) {}
\bA  node (b0) {}
\bB  node (b1) {}
\bC  node (b2) {};
\draw[black] 
\aA -- \aB -- \aC -- \aA
;
\draw[black] 
\bA -- \bB -- \bC --\bA
;
\draw[black] 
\aA -- \bA --
\aB -- \bB --
\aC -- \bC -- \aA 
;
\end{tikzpicture}
\]
compute the magnitude homology $\MH_{n,\ell}(\hat G_K,\ZZ),$ and show that $\hat G_K$ satisfies property $(\mathcal V_2)$ but it is not diagonal. 

Links of all vertices of $K$ are contractible. Links of $6$ edges consist of one point, and links of other $6$ edges consist of two points. Links of triangles are empty. So the reduced homology of all of them is concentrated in the top dimension. Therefore by Proposition \ref{proposition:triv_reduced} the non-diagonal part is defined by homology of $K,$ concentrated in degree $(3,4)$ and isomorphic to  $\MH_{3,4}(\hat G_K,\ZZ)\cong H_1(K,\ZZ)\cong \ZZ.$

Let us compute $\MH_{n,n}(\hat G_K,\ZZ).$ Since $|K(0)|=6,$ $|K(1)|=12$ and $|K(2)|=6,$ the number of vertices in $\hat G_K$ is $|K|+2=26,$ and the number of arrows is $|K(0)|+2|K(1)|+3 |K(2)|+|K(2)|=54.$ Therefore $\MH_{0,0}(\hat G_K,\ZZ)\cong \ZZ^{26}$ and $\MH_{1,1}(\hat G_K,\ZZ)\cong \ZZ^{54}.$
Let us compute $D(2)$ and $D(3):$
\begin{equation}
D(2)=|K(1)|+3|K(2)|=30, \hspace{1cm} D(3)=|K(2)|=6.   
\end{equation}
Links of $6$ edges consist of one point, and links of other $6$ edges consist of two points. Therefore   $\MH_{2,2}(\hat G_K,\ZZ)=\ZZ^{D(2)}\oplus \ZZ^{6} =\ZZ^{36},$ and $\MH_{3,3}(\hat G_K,\ZZ)=\ZZ^{D(3)}=\ZZ^6.$  So the magnitude homology groups  $\MH_{n,\ell}(\hat G_K,\ZZ)$ are free abelian groups of ranks listed in the following table.  
\begin{equation}
\begin{tabular}{c|ccccc}
 & 0 & 1 & 2 & 3 &   \\ \hline
0 & 26 & & & &  \\
1 & & 54 & & &\\
2 & &    & 36 & &\\
3 & &    &    & 6 &\\
4 & &    &    & 1 &\\
\end{tabular}
\end{equation}
\end{example}

\subsection{Triangulations of manifolds}

A triangulation of a topological manifold $M$ is a simplicial complex $K$ together with a homeomorphism $|K|\cong M,$ where $|K|$ is the geometric realization.

\begin{theorem}\label{th:manifold}
Let $M$ be a topological manifold with boundary, $K$ be a  triangulation of $M$ and $\KK$ be a commutative ring. 
Then for $n\neq \ell$ we have 
\begin{equation}
\MH_{n,\ell}(\hat G_K,\KK) \cong 
\begin{cases}
\bar H_{n-2}(M,\KK),&    \ell-2=\dim(M),\\
0,& \text{otherwise.}
\end{cases}
\end{equation}
\end{theorem}
\begin{proof}
For any simplicial complex $K,$ any simplex $\sigma \in K$ and any point from the interior of the geometric realisation of the simplex $x\in |\sigma|^\circ  \subseteq |K|,$ the local homology of $|K|$ at the point $x$ is isomorphic to the shifted homology of the link of $\sigma$ \cite[Lemma 5.2.2]{shastri2016basic}
\begin{equation}\label{eq:iso:local:homo}
H_i(|K|,|K|\setminus \{x\}) \cong \bar H_{i-\dim(\sigma)-1}(\Lk(\sigma)).
\end{equation}

Without loss of generality, we assume that $|K|$ is a topological manifold of dimension $d.$ Then by excision theorem for any $x\in |K|$ the group $H_i(|K|,|K|\setminus \{x\})$ is either isomorphic to  $H_i(D^d, D^d\setminus \{0\})$ or to $H_i(D^d_{\geq 0},D^d_{\geq 0}\setminus \{0\}),$ where $D^d \subseteq \mathbb{R}^d$ is the standard closed disk, and $D^d_{\geq 0} = D^d\cap (\mathbb{R}^{d-1}\times \mathbb{R}_{\geq 0}).$ Combining this with the formula \eqref{eq:iso:local:homo} we obtain that the reduced homology of $\Lk(\sigma)$ is concentrated in the top dimension, for any $\sigma\in K$ (compare with \cite[Th.5.2.4(1)]{shastri2016basic}). Then the statement follows from Proposition \ref{proposition:triv_reduced}.
\end{proof}

\begin{corollary}
Let $K$ be a triangulation of a closed topological manifold $M$ of dimension at least one. Then $\hat G_K$ is diagonal if and only if $M$ is a homology sphere. 
\end{corollary}
\begin{proof} Set $d=\dim(M)\geq 1.$ If $M$ is a homology sphere, then by Theorem \ref{th:manifold} we obtain that $\hat G_K$ is diagonal. Now assume that $\hat G_K$ is diagonal and prove that $M$ is a homology sphere. By Theorem \ref{th:manifold} we obtain that  $\bar H_{n-2}(M,\ZZ)=0$ for $n\neq \ell,$ where $\ell=d+2.$ Substituting $m=n-2,$ we obtain $\bar H_{m}(M,\ZZ)=0$ for $m\neq d.$ Since $d\geq 1,$ we have $\bar H_0(M,\ZZ)=0.$ Therefore, $M$ is connected. If $d=1,$ then in this case we obtain $M=S^1.$ Further we assume that $d\geq 2.$ 

Assume that $M$ is not orientable. Then there is a subgroup of index two in $\pi_1(M)$ \cite[Prop. 3.25]{hatcher2002algebraic}. A subgroup of index two is normal. Therefore  $H_1(M,\ZZ)=\pi_1(M)_{ab}$ is nontrivial, and this gives a contradiction, because $d\geq 2$.

Assume that $M$ is orientable. For a connected orientable closed manifold $M$ of dimension $d$ we have $H_0(M,\ZZ)=H_{d}(M,\ZZ)=\ZZ $  \cite[\S 3.3]{hatcher2002algebraic}. Hence  $M$ is a homology sphere.  
\end{proof}

\begin{example}\label{example:n_0l_0} Here we provide a generalization of 
Example \ref{example:triangle}. 
Let $2\leq n_0 < \ell_0$ and $K$ be a triangulation of the product $M=S^{n_0-2}\times I^{\ell_0-n_0}.$ The homology of $M$ can be computed by the K\"unneth theorem. By Theorem \ref{th:manifold} we obtain that for $n\neq \ell$ we have 
\begin{equation}
\MH_{n,\ell}(\hat G_K,\ZZ) \cong  
\begin{cases}
\ZZ ,& (n,\ell) = (n_0,\ell_0)\\
0, & \text{otherwise}.
\end{cases}
\end{equation}
Therefore, for $n_0\geq 3,$ the digraph $\hat G_K$ satisfies  $(\mathcal V_2)$ but is not diagonal. 
\end{example}

\end{document}